\documentclass{amsart}
\usepackage[T1]{fontenc}
\usepackage{float}
\usepackage{amsmath}
\usepackage{mathrsfs}
\usepackage{amssymb}
\usepackage{graphicx}
\usepackage{mathabx}
\usepackage{bbm}
\usepackage{tipa}
\usepackage[nointegrals]{wasysym}
\usepackage{tensor}
\usepackage{verbatim}
\usepackage{graphicx}
\usepackage{scalerel}
\usepackage{hyperref}
\usepackage{xcolor}
\usepackage{mathtools}
\usepackage[all,arc]{xy}
\usepackage{multicol}
\usepackage{relsize}
\usepackage{tikz-cd}
\usepackage[toc,page]{appendix}

\theoremstyle{plain}
\newtheorem{thm}{Theorem}[section]
\newtheorem{prop}[thm]{Proposition}
\newtheorem{cor}[thm]{Corollary}
\newtheorem{lem}[thm]{Lemma}
\newtheorem{fact}[thm]{Fact}

\newtheorem{ques}{Question}

\newtheorem*{thmA}{Theorem A}
\newtheorem*{thmB}{Theorem B}

\newtheorem*{thmD}{Theorem D}

\newtheorem*{quesC}{Question C}

\newtheorem*{thmF}{Theorem F}

\newtheorem*{ques*}{Question}

\newtheorem*{cor*}{Corollario}

\theoremstyle{definition}
\newtheorem{defi}[thm]{Definition}
\newtheorem*{remE}{Remark E}

\theoremstyle{remark}
\newtheorem{rem}[thm]{Remark}
\newtheorem{example}[thm]{Example}

\numberwithin{equation}{section}

\newcommand{\Z}{\mathbb{Z}}
\newcommand{\Q}{\mathbb{Q}}

\newcommand{\R}{\mathbb{R}}
\newcommand{\F}{\mathbb{F}}

\newcommand{\iid}{\mathrm{id}}

\newcommand{\Hom}{\mathrm{Hom}}

\newcommand{\SL}{\mathrm{SL}}
\newcommand{\vcd}{\mathrm{vcd}}
\newcommand{\ccd}{\mathrm{cd}}
\newcommand{\Hc}{\mathrm{H}_c}
\newcommand{\QG}{\mathbb{Q}[G]}
\newcommand{\dH}{\mathrm{dH}}

\newcommand{\Aut}{\mathrm{Aut}}
\newcommand{\BiG}{\mathrm{\bf Bi}(G)}
\newcommand{\caO}{\mathcal{O}}
\newcommand{\flatrk}{\mathrm{flat}\text{-}\mathrm{rk}}
\newcommand{\algrk}{\mathrm{alg}\text{-}\mathrm{rk}}

\newcommand{\caS}{\mathcal{S}}
\newcommand{\caA}{\mathcal{A}}

\newcommand{\caM}{\mathcal{M}}

\newcommand{\euV}{\mathscr{V}}
\newcommand{\euE}{\mathscr{E}}
\newcommand{\caP}{\mathscr{P}}
\newcommand{\euT}{\mathscr{T}}

\newcommand{\ind}{\mathrm{ind}}
\newcommand{\eue}{\mathbf{e}}

\newcommand{\mx}{\mathbf{mx}}

\newcommand{\caC}{\mathcal{C}}
\newcommand{\frp}{\mathfrak{p}}

\newcommand{\DeltaDav}{\Delta_{\mathrm{Dav}}}
\newcommand{\caR}{\mathcal{R}}

\newcommand{\Res}{\mathrm{Res}}

\newcommand{\btd}{\text{\larger[3.3]{$\blacktriangledown$}}}
\newcommand{\btu}{\blacktriangle}

\newcommand{\trt}{\btd\!/\!\btu}

\newcommand{\caG}{\mathcal{G}}
\newcommand{\caT}{\mathscr{T}}
\newcommand{\caV}{\mathscr{V}}
\newcommand{\caE}{\mathscr{E}}

\newcommand{\Gl}{\mathrm{GL}}
\newcommand{\image}{\mathrm{im}}
\newcommand{\spn}{\mathrm{span}}

\newcommand{\boG}{\mathbf{G}}
\newcommand{\ark}{\algrk}

\newcommand{\argu}{\hbox to 1.5ex{\hrulefill}}  

\newcommand{\bianca}[1]{\textcolor{blue}{#1}}

\begin{document}
\title[Some invariants of t.d.l.c.~groups: cohomology and combinatorics]{Some invariants of totally disconnected locally compact groups: cohomology and combinatorics}
\begin{abstract}  The paper investigates two invariants for totally disconnected locally compact~(t.d.l.c.) groups: the number of ends and the rational discrete cohomological dimension.
 Firstly, for a compactly generated t.d.l.c.~group~$G$, 
it is shown that its number of ends can be expressed in terms of the low-dimensional rational discrete cohomology of~$G$,
generalising a notable theorem by E.~Specker to the t.d.l.c.~setting. Secondly, we focus on a t.d.l.c.~group~$G$ acting under suitable conditions on a building $\Delta$, showing that both the number of ends and the rational cohomological dimension of~$G$ are related to those of the Weyl group associated to~$\Delta$. In special cases, we compare the rational discrete cohomological dimension of~$G$ to its flat-rank, and in the appendix we provide examples for which these two invariants coincide. Finally, for t.d.l.c.~groups with rational discrete cohomological dimension one that act properly and chamber-transitively on locally finite buildings, we explicitly construct a proper cocompact action on a tree out of the action on the building.
All the results of independent interest in Coxeter theory have been included in the appendix.
\end{abstract}
\author{Ilaria Castellano}
\address{
Dipartimento di Matematica e Applicazioni, Universit\`a degli Studi di
Milano-Bicocca, Via Roberto Cozzi 55, I-20125 Milano, Italy}
\email{ilaria.castellano@unimib.it}
\author[Bianca Marchionna]{Bianca Marchionna}
\address{
Université catholique de Louvain\\
Research Institute in Mathematics and Physics\\
Chemin du Cyclotron 2, 1348 Louvain-la-Neuve, Belgium}
\email{bianca.marchionna@uclouvain.be}
\author{Thomas Weigel}
\address{Dipartimento di Matematica e Applicazioni, Universit\`a degli Studi di
Milano-Bicocca, Via Roberto Cozzi 55, I-20125 Milano, Italy}
\email{thomas.weigel@unimib.it}

\date{\today}
\maketitle

\section{Introduction}\label{s:intro}
There are several invariants and properties of groups that admit a cohomological interpretation or can be related to suitable cohomological invariants. For instance, E.~Specker~\cite[Satz~IV]{specker} proved that the number of ends~$e(G)$ of a finitely generated group~$G$ can be expressed in terms of the rank of the cohomology groups of $G$ by
\begin{equation}\label{eq:specker}
    e(G)=1-\mathrm{rk}_{\Z}\mathrm{H}^0(G,\Z[G])+\mathrm{rk}_{\Z}\mathrm{H}^1(G,\Z[G]).
\end{equation}

Analogous phenomena occur in the context of totally disconnected locally compact (= t.d.l.c.) groups, 
as demonstrated by recent advances. Specifically, the property of having more than one end can be detected by the first cohomology group; see~\cite[Theorem~A*]{cast:cd1}. Furthermore, there are strong indications that the t.d.l.c.~groups of cohomological dimension one should precisely be the fundamental groups of graphs of profinite groups~\cite{cmw:stasw}. The purpose of this paper is to establish new results that fit this general framework. The first main result we prove consists in generalising Specker's formula~\eqref{eq:specker} to the number of ends~$e(G)$ (cf.~\cite[Definition~3.12]{km:cayley} and Section~\ref{ss:ends}) of a compactly generated t.d.l.c.~group as follows.
\begin{thmA}[\protect{cf.~Theorem~\ref{thm:end2}}]\label{thmJ}
    Let $G$ be a compactly generated t.d.l.c.~group. Then 
    \begin{equation*}
        e(G)=1-\dim_\Q\dH^0(G,\BiG)+\dim_\Q\dH^1(G,\BiG),
    \end{equation*}
  where $\BiG$ is the rational discrete standard $\Q[G]$-bimodule\footnote{The definition of $\BiG$ is recalled in~\eqref{eq:biG}. It is non-canonically isomorphic to the space of continuous locally constant functions from~$G$ to~$\Q$ with compact support \cite[\S 4.7]{cw:qrat}.} and $\dH^k(G,\argu)$, $k\in \Z_{\geq 0}$, denote the \emph{rational discrete cohomology functors} as introduced in~\cite{cw:qrat}.
    \end{thmA}
    The core idea of the proof is drawn from Specker’s original argument; however, adapting it to the t.d.l.c.~context requires further non-trivial technical effort. Note that the cohomology theory we use for t.d.l.c.~groups is the one developed for discrete modules over the rational group algebra $\QG$~(cf.~\cite{cw:qrat} and Section~\ref{ss:coh}). Within this theory, an invariant that will be extensively used in this work is the \emph{rational discrete cohomological dimension} of a t.d.l.c.~group~$G$, defined as
$$\ccd_\Q(G):=\sup\big\{k\in \Z_{\geq 0}\mid \dH^k(G,\argu)\neq 0\big\}$$
(cf.~\cite[Equation~(3.11)]{cw:qrat} or Section~\ref{ss:coh}).
Note that $\ccd_\Q(G)$ coincides with the classical rational cohomological dimension in the case that~$G$ is discrete; see Fact~\ref{fact:discoho}. Our main results concern $\ccd_\Q(G)$, but also other two invariants, which do not explicitly arise from a cohomological framework. These are the number of ends and the {\em flat rank}~$\flatrk(G)$ of~$G$ (cf.~\cite{wil04} and~Section~\ref{s:flatrk}).   

\smallskip

  Typical examples to which our results apply are t.d.l.c.~groups acting properly and cocompactly on the chambers of a locally finite building\footnote{In this paper, buildings are considered as chamber systems (cf.~Definition~\ref{defi:building}). The terminology is adapted accordingly: e.g., locally finite buildings are buildings with finite panels (cf.~Section~\ref{sss:build}).}.
 E.g., semisimple algebraic groups defined over a non-Archimedean local field~\cite{brutit1,brutit2}, complete Kac--Moody groups defined over a finite field~\cite{rr06}, automorphism groups of locally finite right-angled buildings and certain universal subgroups \cite{cap:simple, dmss:uniRA}.
 The benefit of studying t.d.l.c.~groups acting on buildings is that one can often translate the problem to Coxeter groups. For example, if~$G$ acts properly and cocompactly on a locally finite building of type $(W,S)$, we exploit Theorem~\hyperref[thmJ]{A} to give an alternative proof of the fact that~$G$ has one end if and only if~$W$ has one end (cf.~\cite[pp.~23-24]{cmr:KMgrps} and~Proposition~\ref{prop:ecd}). In the appendix, we collect some new results on the rational cohomological dimension of Coxeter groups that will be used in the paper and might be of independent interest in Coxeter theory.

 Moreover, we will make extensive use of the Davis realisation $|\DeltaDav|$ of a building~$\Delta$ (cf.~Section~\ref{sss:dav}). It is a topological space whose compactly supported cohomology $\mathrm{H}_c^\bullet(|\DeltaDav|,\Z)$ has been explicitly computed by M.~W.~Davis, J.~Dymara, T.~Januszkiewicz, J.~Meier and B.~Okun in the main theorem of~\cite{compsup}. These explicit formulae fit perfectly our need in order to show the following: every t.d.l.c.~group acting properly and cocompactly on a locally finite building of type~$(W,S)$ satisfies
\begin{equation}\label{eq:cd}
    \ccd_\Q(G)=\ccd_\Q(W)
    \end{equation}
(cf.~Theorem~\ref{thm:ihbuil}). Equation~\eqref{eq:cd} provides an efficient way for computing $\ccd_\Q(G)$ by taking advantage of the combinatorics of Coxeter groups. For instance, one may use the Bestvina complex (cf.~Section~\ref{sus:vcd}) to compute the rational cohomological dimension of a Coxeter group (cf.~\cite[Remarks~(1)]{bes:vcd}), and so $\ccd_\Q(G)$. In Appendix~\ref{s:hyper} (see also Theorem~\hyperref[thmD]{D} below), we specifically apply this method to Coxeter groups of hyperbolic type.

The equality~\eqref{eq:cd} is not an isolated phenomenon. D.~Degrijse and C.~Mart\'inez-Pérez~\cite{degmp} and, more recently, N.~Petrosyan and T.~Prytu{\l}a~\cite{pepr:bredon} have established an equality analogous to~\eqref{eq:cd} in Bredon group cohomology, even in a more general context than groups acting on buildings.

Another relevant result that fits our purposes is the inequality $$\flatrk(G)\leq \algrk(W)$$ which relates the flat rank of~$G$ to the algebraic rank of $W$ for sufficiently transitive actions of a t.d.l.c.~group $G$ on a building~$\Delta$ of type~$(W,S)$ (see \cite[Theorems~A and~B]{brw:build} and~\cite[Corollary~C]{caphag}).  
Since $\algrk(W)\leq\ccd_\Q(W)$, the equality~\eqref{eq:cd} implies the following:
\begin{thmB}[\protect{cf.~Theorem~\ref{thm:fl-cd}}]\label{thmB}
   Let $\Delta$ be a locally finite building and $G$ a Weyl-transitive closed subgroup of the group of type-preserving automorphisms of $\Delta$. Then
       $ \flatrk(G)\leq\ccd_\Q(G).$
\end{thmB}
If $\Delta$ is thick (cf.~Section~\ref{sss:build}), we also provide an alternative proof of Theorem~\hyperref[thmB]{B} that does not rely on the results in~\cite{brw:build, caphag}. Namely, we show that every t.d.l.c.~group acting properly and Weyl-transitively on a locally finite thick building is \emph{$N$-compact}, i.e., the normaliser of every compact open subgroup is compact (cf.~Proposition~\ref{prop:bcg}). Since the inequality $\flatrk(G)\leq \ccd_\Q(G)$  holds for all  $N$-compact t.d.l.c.~groups $G$ (cf.~\cite[Proposition~3.10(b)]{cw:qrat}), one concludes Theorem~\hyperref[thmB]{B}.
This gives rise to the following  question.

\begin{quesC}\label{quesC}
    For which t.d.l.c.~groups~$G$ does~$\flatrk(G)=\ccd_\Q(G)$ hold?
\end{quesC}
If $G$ is compact then $\flatrk(G)=0=\ccd_\Q(G)$ (cf.~\cite[Proposition~3.7(a)]{cw:qrat}). 
For semisimple algebraic groups over a non-Archimedean local field and complete Kac--Moody groups over a finite field the equality $\flatrk(G)= \algrk(W)$ holds (cf.~\cite[Theorem~17]{brw:build} and~\cite[Corollary~C]{caphag}). Hence, the answer to Question~\hyperref[quesC]{C} for the latter classes of groups is provided by a comparison of the invariants $\algrk(W)$ and $\ccd_\Q(W)$. In~\cite[\S 6.11]{cw:qrat} it has been already pointed out that
$$\flatrk(G)=\algrk(W)=\ccd_\Q(W)=\ccd_\Q(G)$$
whenever $G$ is semisimple algebraic over a {non-Archimedean} local field. In this case the Weyl group $(W,S)$ is of affine type and so $\algrk(W)=\ccd_\Q(W)$.
Within the class of Coxeter groups of hyperbolic type we prove the following: 
\begin{thmD}[\protect{cf.~Propositions~\ref{prop:vcdhyp} and~\ref{prop:noncomp}}]\label{thmD}
     For a Coxeter group $(W,S)$ of hyperbolic type the following hold:
     \begin{enumerate}
         \item[(a)] If $(W,S)$ is compact, then 
 $\ark(W)=1$ and $\ccd_\Q(W)=|S|-1\geq 2$.
 \item[(b)] If $(W,S)$ is non-compact, then $\ark(W)=\ccd_\Q(W)=|S|-2.$
     \end{enumerate}
\end{thmD}
\begin{remE}
    Combining Proposition~\ref{prop:algcd} and Theorem~\hyperref[thmD]{D}, we answer Question~\hyperref[quesC]{C} for the geometric completion of every group with a root group datum of hyperbolic type~$(W,S)$ and finite root groups. Namely, for such a group~$G$, one has~$\flatrk(G)=\ccd_\Q(G)$ if and only if~$(W,S)$ is non-compact hyperbolic.
\end{remE}
The last section of the paper focuses on t.d.l.c.~groups of cohomological dimension one. In the discrete case, M.~Dunwoody characterised groups of rational cohomological dimension at most one as fundamental groups of graphs of finite groups \cite{dun:acc}. It is still unknown whether Dunwoody's characterisation extends to non-discrete t.d.l.c.~groups~(see \cite[Question~D]{cmw:stasw} for more details on this open problem). If the t.d.l.c.~group $G$ admits a suitable action on a locally finite building, this problem has an affirmative answer:
\begin{thmF}[\protect{cf.~Theorem~\ref{thm:cd1}}]
    Let $G$ be a t.d.l.c.~group acting chamber-transitively  on a locally finite building $\Delta$ with compact open stabilisers. If $\ccd_\Q(G)=1$ then $G$ decomposes as fundamental group of a finite tree of profinite groups. 
\end{thmF}
By~\eqref{eq:cd}, if $\ccd_\Q(G)=1$ then the Coxeter group $(W,S)$ is infinite and virtually free. This simple observation allows us to construct a proper continuous action on a tree out of the given action on the building. One often spots tree-like structures in buildings and polyhedral complexes (see for instance~\cite[\S 5]{basw} and~\cite[Proof of Proposition~2.1]{dyja}. Our method is based on a construction due to F.~Haglund and F.~Paulin~\cite[\S 4.2]{hp:arb}, which we extend in Section~\ref{s:hp} from free products with amalgamation to visual graph of groups decompositions of $(W,S)$ (cf.~Definition~\ref{def:visualdec}). 
\subsection*{Acknowledgement} 
The authors thank P.-E.~Caprace for valuable comments that improved a first draft. The authors also thank the referees for several valuable suggestions that improved the correctness, the exposition, the organisation and the clarity of the paper. BM's research was partially funded by F.R.S..FNRS through a postdoctoral researcher grant.
IC and BM were partially supported by the Deutsche Forschungsgemeinschaft (DFG, German Research Foundation) – Project-ID 491392403 – TRR 358. BM was also partially supported by RTG 2229 “Asymptotic Invariants and Limits of Groups and Spaces”.
TW gratefully acknowledges financial support by the PRIN2022 “Group theory and its applications”.
 IC~and TW are members of “INdAM - GNSAGA Project”, CUP E53C25002010001. All authors are members of the Gruppo Nazionale per le Strutture Algebriche,
Geometriche e le loro Applicazioni (GNSAGA), which is part of the Istituto
Nazionale di Alta Matematica (INdAM).

\subsection*{AI disclosure}
The research presented in this paper and the writing-up have been carried out without any use of~AI.

\section{Preliminaries}
\subsection{Graphs}
\label{ss:naiGr}
For a set~$X$ we denote by~$\caP_2(X)$ the set of all subsets of~$X$ of cardinality~$2$.
An \emph{undirected graph} $\Gamma=(V,E)$ consists of a 
set of
vertices $V=V(\Gamma)$ together with a set of edges
$E=E(\Gamma)\subseteq\caP_2(V)$. 
A {\em subgraph $(U,F)$} of $\Gamma$ is an undirected graph such that~$U\subseteq V$ and~$F\subseteq E$. An undirected graph $(V,E)$ is said to be 
\emph{complete} if $E=\caP_2(V)$. A complete
subgraph $(U,F)$ of an undirected graph $(V,E)$ satisfying
$|U|=n$ is said to be an \emph{$n$-clique}. 

A subgraph $(U,F)$ of an undirected graph
$\Gamma=(V,E)$ is said to be \emph{induced} if
$F=\caP_2(U)\cap E$. For instance, any $n$-clique is an induced subgraph. 
Given $U\subseteq V$, denote by~$\ind_\Gamma(U)$ the induced subgraph of~$\Gamma$ with vertex set~$U$. 
Given a subgraph $\Lambda$ of $\Gamma$, one denotes by~$\Gamma-\Lambda$ the subgraph 
$$\big(V(\Gamma)\setminus V(\Lambda), \{\{x,y\}\in E(\Gamma)\mid x,y\not\in V(\Lambda)\}\big).$$
If $\Lambda$ consists of a $1$-vertex graph with vertex-set $\{v\}$, we write $\Gamma-v$ instead of~$\Gamma-\Lambda$.

A \emph{path} in~$\Gamma$ is a finite sequence of vertices $(v_0, \dots, v_n)$, $n\geq 0$, such that $\{v_i,v_{i+1}\}\in E$ for every $0\leq i\leq n-1$. Every path $\frp=(v_0,\dots, v_n)$ determines a subgraph~$(V(\frp),E(\frp))$ of~$\Gamma$ by setting $V(\frp)=\{v_0,\dots, v_n\}$ and $E(\frp)=\big\{\{v_{i},v_{i+1}\}\mid 0\leq i\leq n-1\big\}$.
The path $(v_0,\dots, v_n)$ is said to be \emph{simple} if~$n=0$ or $v_0=v_n$ and~$v_i\neq v_j$ for all distinct~$i,j\in\{0,\dots, n\}$ with~$(i,j)\neq (0,n)$. A \emph{cycle} is a simple path $(v_0,\dots, v_n)$ with~$v_0=v_n$ and $v_{n-1}\neq v_1$.

An undirected graph $\Gamma=(V,E)$ is said to be \emph{chordal}
if the only induced cycles in $\Gamma$ have $3$ vertices.

\smallskip
Every undirected graph
$\Gamma=(V,E)$ defines a combinatorial graph $\Vec{\Gamma}=(V,\Vec{E})$ in the sense
of J-P.~Serre (cf.~\cite[\S I.2.1, Definitions~1~and~5]{ser:trees}), where
\begin{equation}
\label{eq:SerGr}
\Vec{E}=\{(v,w),(w,v)\mid\{v,w\}\in E\},
\end{equation}
the origin map $o\colon\Vec{E}\to V$ is the projection on the first coordinate, the terminus map $t\colon\Vec{E}\to V$
is the projection on the second coordinate,
and the edge-inversion map $\bar{\cdot}\colon\Vec{E}\to\Vec{E}$
interchanges the first and second coordinates.
On the other hand, every combinatorial graph
$\Gamma=(\caV,\caE)$ in the sense of J-P.~Serre defines
an undirected graph $\ddot{\Gamma}=(\caV,\ddot{\caE})$ given by
\begin{equation}
\label{eq:NGr}
\ddot{\caE}=\big\{\{o(\eue),t(\eue)\}\mid \eue\in\caE\big\}.
\end{equation}

\subsection{Coxeter groups and associated diagrams}
\label{ss:coxgr}
For a finite set~$S$, a \emph{Coxeter matrix} $[m_{st}]_{s,t\in S}$ is a symmetric matrix with diagonal entries equal to $1$ and with non-diagonal entries in $\Z_{\geq 2}\cup\{\infty\}$.
For such a matrix $M$ one may associate  the {\it Coxeter group} $(W,S)$, which is the group~$W$ generated by~$S$ subject to the relations $(st)^{m_{st}}=1$, whenever $m_{st}\neq\infty$. It turns out that~$m_{st}$ coincides with the order of~$st$ in~$W$. 
 A {\em special subgroup} $W_J$ of $(W,S)$ is a subgroup generated by a subset $J$ of $S$. Notice that $(W_J,J)$ is a Coxeter group for every $J\subseteq S$ (cf.~\cite[Theorem~5.5(a)]{hum:ref}).
A Coxeter group $(W,S)$ is said to be \emph{spherical} if $W$ is finite. One says that $J\subseteq S$ is \textit{spherical} if $(W_J,J)$ is spherical. For further details, the reader is referred to~\cite[\S 5]{hum:ref}.

Classically, to each Coxeter group $(W,S)$ one associates the so-called \emph{Coxeter diagram} $\Gamma(W,S)$, which is the edge-labelled undirected graph with set of vertices $S$, set of edges
$E(\Gamma(W,S))=\big\{\{s,t\} \mid s,t\in S,\, m_{st}\neq 2\big\}$ and edge-labelling $m\colon E(\Gamma(W,S))\to \Z_{\geq 3}\cup\{\infty\}$, $m(\{s,t\}):=m_{st}$. Usually, one omits the label of all edges $\{s,t\}$ in $\Gamma(W,S)$ satisfying $m_{st}=3$.
Sometimes, it is convenient to consider another graph associated to $(W,S)$, the \emph{presentation diagram}~$\Gamma_\infty(W,S)$ of~$(W,S)$ (cf.~\cite[p.~174]{mt:visualdec}). It is the edge-labelled undirected graph with vertex set $S$ and edge set 
$E(\Gamma_\infty(W,S))=\big\{\{s,t\} \mid s,t\in S,\, m_{st}\neq \infty\big\}$, in which every edge $\{s,t\}$ is labelled with $m_{st}$.
Each of the two edge-labelled undirected graphs, $\Gamma(W,S)$ and $\Gamma_\infty(W,S)$,  uniquely determines the Coxeter matrix $[m_{st}]_{s,t\in S}$ and thus the pair~$(W,S)$. 
\subsection{Decompositions of Coxeter groups}\label{ss:infdec}
Let $(W,S)$ be a Coxeter group. Two subsets~$S_\btd, S_\btu\subseteq S$, with~$S_\bullet:=S_\btd\cap S_\btu$, satisfying 
\begin{itemize}
    \item[(i)] $S=S_\btd\cup S_\btu$; and,
    \item[(ii)] $m_{st}=\infty$ for all $s\in S_\btd\setminus S_\bullet$ and $t\in S_\btu\setminus S_\bullet$;
\end{itemize}  
will be called an \textit{$\infty$-decomposition of~$(W,S)$}. An $\infty$-decomposition~$S=S_\btd \cup S_\btu$ of~$(W,S)$ is said to be \emph{non-trivial} if~$S_{\btd}\cap S_{\btu}$ is a proper subset of~$S_\btd$ and~$S_\btu$. 

The existence of an $\infty$-decomposition~$S=S_\btd\cup S_\btu$ is equivalent to the fact that~$W$ splits as an amalgamated free product~$W=W_{\btd}\ast_{W_\bullet}W_\btu$, where~$W_\times=\langle S_\times\rangle$ is the subgroup of~$W$ generated by~$S_\times\subseteq S$ for every~$\times\in\{\btd, \btu,\bullet\}$.

\medskip

Given a Coxeter group $(W,S)$ and a subset $J\subsetneq S$, it is possible to investigate $\infty$-decompositions $S=S_\btd\cup S_\btu$ satisfying $J=S_\btd\cap S_\btu$ considering the
presentation diagram $\Gamma_\infty(W,S)$ (cf.~Section~\ref{ss:coxgr}).
For this purpose, let~$\Xi_J$ be the subgraph of
$\Gamma_\infty(W,S)$ spanned by $S\setminus J$, i.e.,
\begin{equation}
\label{eq:cox2}
\begin{aligned}
\euV(\Xi_J)&=S\setminus J,\\
\euE(\Xi_J)&=\{\{u,v\}\in\caP_2(S\setminus J)\mid m_{uv}\neq\infty\}.
\end{aligned}
\end{equation}
Then there exists a non-trivial $\infty$-decomposition $S=S_\btd\cup S_\btu$
for $(W,S)$ satisfying~$S_\btd\cap S_\btu=J$ if and only if
$\Xi_J$ is disconnected. Indeed,
let $L\subseteq S\setminus J$ be the set of vertices of a connected component of $\Xi_J$ and let $V=S\setminus (J\cup L)$. Then 
$S_\btd=L\cup J$ and $S_\btu=V\cup J$ define a non-trivial $\infty$-decomposition
of $(W,S)$. Consequently, one has the following.
\begin{fact}
\label{fact:cox}
Let $(W,S)$ be a Coxeter group. For every $J\subsetneq S$, the following are equivalent:
\begin{itemize}
\item[(i)] there exists a non-trivial $\infty$-decomposition 
$S=S_\btd\cup S_\btu$ of $(W,S)$ satisfying $J=S_\btd\cap S_\btu$;
\item[(ii)] $\Xi_{J}$ has at least $2$ connected components.
\end{itemize}
\end{fact}

We now state a result due to M.~W.~Davis (cf.~\cite[Theorem~8.7.2]{dav:book}) in terms of \emph{spherical $\infty$-decompositions} of a Coxeter group $(W,S)$, i.e., $\infty$-decompositions $S=S_\btd\cup S_\btu$ of $(W,S)$ in which $S_\btd\cap S_\btu$ is a spherical subset of $S$. For the reader's convenience, we give a complete proof of the statement here.
\begin{thm}
\label{thm:dav}
Let $(W,S)$ be a Coxeter group.
Then the following are equivalent:
\begin{itemize}
\item[(i)] the group $W$ has more than one end;
\item[(ii)] $(W,S)$ has a non-trivial spherical $\infty$-decomposition. 
\end{itemize}
\end{thm}
\begin{proof} 
The implication (ii)$\Rightarrow$(i) follows by Stallings' decomposition theorem. Indeed, assume that $(W,S)$ has a non-trivial spherical $\infty$-decomposition $S=S_\btd\cup S_\btu$. Then, it is
evident from the Coxeter presentation that $W$ is a non-trivial free product
of the special subgroups corresponding to $S_\btd$ and $S_\btu$ amalgamated over the special subgroup
generated by $S_\btd\cap S_\btu$ (cf.~\cite[\S 4.1, Remarques~(1)]{hp:arb}).

Conversely, suppose that the group~$W$ has more than one end. Then one must have $\mathrm{H}^1(W,\Q[W])\neq0$ (cf.~\cite[Chapter~IV, Theorem~6.10]{dd:stall}). 
By~\cite[Corollary~8.5.2]{dav:book},
\begin{equation}
\label{eq:davis1}
\mathrm{H}^1(W,\Q[W])\cong 
\bigoplus_{J\subset S,\atop 
{J\text{ spherical}}} \Q[W^J]\otimes \mathrm{H}^1(K,K^{S\setminus J};\Q),
\end{equation}
where $K$ is the fundamental chamber of $|\Sigma(W,S)_{\mathrm{Dav}}$| (cf.~\cite[p.~126]{dav:book}) and 
$\Q[W^J]$ is the $\Q$-vector space spanned by the set~$W^J$ defined in~\cite[Definition~4.7.4]{dav:book}.
In particular, there exists $J\subseteq S$ spherical such that
$\mathrm{H}^1(K,K^{S\setminus J};\Q)\neq0$.
By the Universal Coefficient theorem, 
$\mathrm{H}_1(K,K^{S\setminus J};\Q)\neq0$.
As $K$ is contractible (cf.~\cite[Lemma~7.2.5(i)]{dav:book}), one has a short exact sequence
\begin{equation}
\label{eq:davis2}
0\longrightarrow \mathrm{H}_1(K,K^{S\setminus J};\Q)\longrightarrow
 \mathrm{H}_0(K^{S\setminus J},\Q)\longrightarrow \Q\longrightarrow 0.
\end{equation}
Thus, $\mathrm{H}_1(K,K^{S\setminus J};\Q)\neq0$ implies that $K^{S\setminus J}$
has more than one path-connected component.
By definition,
$K^{S\setminus J}$ is the union of the connected spaces $K_s$, $s\in S\setminus J$, where each $K_s$ is the geometric realisation of the poset $\{T\subseteq S \mid T\text{ spherical, }s\in T\}$
 (cf.~\cite[\S 7.2]{dav:book}). Moreover, $K_s\cap K_t\neq\emptyset$ if and only if $m_{st}<\infty$. Hence, since  $K^{S\setminus J}$
has at least two connected components, the number of connected components of $\Xi_J$ is at least two. Now Fact~\ref{fact:cox} yields the claim. 
\end{proof}
Since $W$ is finitely generated, the number of ends~$e(W)$ of~$W$ satisfies $e(W)\in\{0,1,2,\infty\}$ by the Freudenthal--Hopf theorem. Moreover, $e(W)=0$ if and only if $W$ is finite. 
The following is a direct consequence of Theorem~\ref{thm:dav}.
\begin{cor}\label{cor:cends}
   Let $(W,S)$ be a Coxeter group with Coxeter matrix $[m_{st}]_{s,t\in S}$. 
  \begin{itemize}
     \item[(i)] Let $(W,S)$ be right-angled, i.e., $m_{st}\in\{2,\infty\}$ for all $s\neq t$. Then $e(W)\geq 2$ if and only if there exists a non-trivial $\infty$-decomposition  $S=S_\btd\cup S_\btu$ such that $m_{st}\leq 2$ for all $s,t\in S_\btd\cap S_\btu$.
     \item[(ii)] If $m_{st}<\infty$ for all $s,t\in S$, then $e(W)\leq 1$. In particular,
     \begin{itemize}
     \item[(iia)] if $(W,S)$ is a Coxeter group of hyperbolic type with $|S|\geq 4$ (cf.~Appendix~\ref{s:hyper}), then $e(W)=1$;
     \item[(iib)] if $(W,S)$ is affine then $e(W)\in\{1,2\}$. Moreover, $e(W)=2$ if and only if $W\cong D_\infty=C_2\ast C_2$.
	\end{itemize}
 \end{itemize}
\end{cor}
 Note that the conclusions of Theorem~\ref{thm:dav} correspond to the equivalence $(1)\Leftrightarrow(2)$ of \cite[Corollary~16]{mt:visualdec}. Their proof is based on the notion of ``visual graph of groups decomposition'', which generalises $\infty$--decompositions from 1-segment graphs to arbitrary graphs of groups. Since this notion will be needed in the following, we recall it below.
\begin{defi}[{\cite[p.~174]{mt:visualdec}}]\label{def:visualdec}
     A {\em graph of special subgroups} $(\caG,\Lambda)$ of a Coxeter group $(W,S)$ based on a connected Serre graph~$\Lambda$ is a graph of groups given as follows. Let $S_\ast\colon \euV(\Lambda)\sqcup \euE(\Lambda)\to \mathscr{P}(S)$ be a map assigning to each $\lambda\in \euV(\Lambda)\sqcup \euE(\Lambda)$ a subset $S_\lambda\subseteq S$  satisfying~$S_e\subseteq S_{t(e)}$ for every~$e\in \euE(\Lambda)$. For every $\lambda\in \euV(\Lambda)\sqcup \euE(\Lambda)$ the group $\caG_\lambda$ coincides with the special subgroup $W_{S_\lambda}$ of $(W,S)$.  Moreover,  for every $e\in \euE(\Lambda)$ the  monomorphism $\caG_{e}\hookrightarrow \caG_{t(e)}$ is 
     given by the inclusion map. Denote by~$\pi_1(\caG,\Lambda)$ the fundamental group of~$(\caG,\Lambda)$ (se~\cite[\S I.5.1]{ser:trees}).
     
     In case that the group homomorphism $\tau\colon\pi_1(\caG,\Lambda)\to W$ induced by the inclusion maps $W_{S_\lambda}\hookrightarrow W$, $\lambda\in \euV(\Lambda)\sqcup \euE(\Lambda)$, is an isomorphism, the pair $(\caG,\Lambda)$ is said to be a \emph{visual graph of groups decomposition} (or simply a \emph{visual decomposition}) of $(W,S)$.
\end{defi} 
\begin{rem}\label{rem:visdec}
    If $(\caG,\Lambda)$ is a visual decomposition of $(W,S)$, then $\Lambda$ must be a tree. Indeed, if $\Lambda$ is not a tree, there is an epimorphism from $\pi_1(\caG,\Lambda)$ to~$\Z$ produced by expressing~$\pi_1(\caG,\Lambda)$ as a certain HNN extension (cf.~\cite[Example~16.4(4)]{bogo:gt}). Since $W$ is generated by involutions,  every homomorphism from $W$ to $\Z$ must be trivial.
\end{rem}
\subsection{The virtual cohomological dimension of Coxeter groups}\label{sus:vcd}
Let $\Gamma$ be a virtually torsion-free group. A result due to J-P.~Serre (cf.~\cite[\S 1.3, Proposition~5]{ser:coho}) states that all finite-index torsion-free subgroups of~$\Gamma$ have the same cohomological dimension over a non-zero commutative unital ring~$R$, which is then defined to be the {\em virtual cohomological dimension $\vcd_R(\Gamma)$ of $\Gamma$ over $R$} (see also~\cite[\S II.5.4]{bieri}). If the cohomological dimension~$\ccd_R(\Gamma)$ of~$\Gamma$ is finite, then $\vcd_R(\Gamma)=\ccd_R(\Gamma)$ (the proof of~\cite[\S VIII.2, Proposition~2.4(a)]{br:coho} extends verbatim to~$R$). From now on $\vcd_\Z(\argu)$ will be denoted by $\vcd(\argu)$.
In 1993, M.~Bestvina (cf.~\cite[Theorem at p.~21]{bes:vcd}) provided an effective method to compute $\vcd_R(W)$ of a finitely generated Coxeter group $(W,S)$, where $R$ is either the integers or a prime field. The construction, recalled below, produces a complex $B_R(W,S)$ of dimension $\vcd_R(W)$ that is $R$-acyclic, i.e.,  the reduced homology
\(\tilde{\mathrm H}_i(B_R(W,S); R)\) vanishes for all degrees $i\geq 0$. We now briefly recall its construction, after providing some preliminary definitions.

Following \cite[p.~350]{hat:at}, a {\em polyhedron} in $\mathbb R^n$
is a subspace
that is the union of finitely many convex polyhedra, each of which is a compact set
obtained by intersecting finitely many half-spaces defined by linear inequalities of the
form $\sum_i a_ix_i\leq b$.
Recall that a polyhedron can always be triangulated as a locally finite simplicial complex (cf.~\cite[p.~12]{rs_piecewise}). 

Let~$(W,S)$ be a Coxeter group and denote by~$\mathfrak S$ the family of all spherical subsets of~$S$ ordered by inclusion. The complex~$B_R(W,S)$, that we call the \emph{Bestvina complex of~$(W,S)$ over~$R$}, is constructed recursively as follows: For every maximal element~$J\in\mathfrak S$ define $P_J$ to be a point. Let $J\in \mathfrak S$ be non-maximal and assume that $P_{J'}$ has been constructed for all $J'\in\mathfrak S$ such that $J'\supset J$. Define $P_J$ to be an $R$-acyclic polyhedron containing $\bigcup_{_{J'\supset J}} P_{J'}$ of the least possible dimension. The construction ends once $P_\emptyset$ has been built. For an example see Figure~\ref{fig:best}.

Following~\cite[p.~20]{bes:vcd}, from $B_R(W,S)$ one constructs an $R$-acyclic polyhedral complex  $X_R(W,S):= W\times B_R(W,S)/\sim$ where $$(w_1,x_1)\sim(w_2,x_2)\Leftrightarrow x_1=x_2\ \text{and}\ w_1^{-1}w_2\in\langle s\in S\mid x_1\in P_{\{s\}}\rangle.$$ 
The group~$W$ acts on~$X_R(W,S)$ cocompactly and properly discontinuously as a reflection group. By construction, the dimension of $X_R(W,S)$ coincides with the dimension of~$B_R(W,S)$.

Recall that, since $W$ is of type $\mathrm{FP}_\infty$ over $\Q$ (cf.~\cite[\S~II.2]{bieri}), the rational cohomological dimension of $W$ satisfies
\begin{equation}\label{eq:cd(W)}
   \ccd_\Q(W)=\max\{n\geq0\mid \mathrm{H}^n(W,\Q[W])\neq0\}
  \end{equation} 
(cf.~\cite[Chapter~VIII, Proposition~6.7]{br:coho}). In Appendix~\ref{s:hyper} the Bestvina complex $B_R(W,S)$ is used to compute the rational cohomological dimension of Coxeter groups of hyperbolic type. Note that, for every Coxeter group $W$, it is also possible to use~$X_\Z(W,S)$ to deduce the well-known inequality: \begin{equation}\label{eq:cd_ineq}
    \ccd_\Q(W)\leq\vcd(W)=\dim(B_\Z(W,S))\leq |S|-1<\infty.
    \end{equation}
    Indeed, tensoring the augmented cellular complex associated to~$X_\Z(W,S)$ with~$\Q$, one obtains a projective $\Q[W]$-resolution of~$\Q$ of length at most the dimension of~$X_\Z(W,S)$. Note that for an affine Coxeter group $W$ of rank~$\ell$ all the inequalities in~\eqref{eq:cd_ineq} are indeed equalities: $W$ is virtually isomorphic to $\Z^{\ell-1}$ and therefore~$\ccd_\Q(W)=\ell-1$.

\subsection{Buildings}\label{sss:build}
Let $(W,S)$ be a Coxeter group. Following~\cite[p.~224]{ab:build}, a {\em chamber system} over $S$ is a non-empty set $\caC$ (whose elements are called {\em chambers}) together with a family of equivalence relations $(\sim_s)_{s\in S}$ on $\caC$ indexed by~$S$. The equivalence classes with respect to $\sim_s$ are called {\em $s$-panels}. {Given~$c\in \caC$, we denote by~$[c]_s$ the $s$-panel containing it. Two chambers~$c$ and $d$ are {\em $s$-adjacent} if they are contained in the same $s$-panel and {\em adjacent} if they are $s$-adjacent for some $s\in S$. A {\em gallery} of length~$n$ connecting $c_0$ and $c_n$ is a sequence $\gamma:c_0,\dots,c_n$ of $n+1$ chambers such that $c_{i-1}$ is adjacent to $c_i$ for all $i=1,\dots,n$. The {\em type} of the gallery $\gamma:c_0,\dots,c_n$ is the tuple $(s_1,\dots,s_n)$ of elements of~$S$ such that~$c_{i-1}$ and~$c_i$ are $s_i$-adjacent for every~$i=1,\dots,n$. A chamber system is said to be \emph{connected} if any two chambers can be connected by a gallery.
Moreover, a chamber system is said to be \emph{thick} if all $s$-panels have at least three chambers for every~$s\in S$, and \emph{locally finite} if all $s$-panels are finite for every~$s\in S$.

\begin{rem}\label{rem:chamgr}
    Every chamber system can be viewed as a graph~$(\caC, E(\caC))$ with coloured edges: the vertices are the chambers and $s$-adjacent chambers are connected by an edge of colour~$s$.\end{rem}

The following fact is well-known to experts but, since its proof is scattered over several references, we record it here.
\begin{prop}\label{fact:Gunimod}
 Let~$G$ be a locally compact group that acts transitively and with compact open stabilisers on a connected locally finite chamber system~$(\caC,(\sim_s)_{s\in S})$. Assume that the $G$-action is \emph{type-preserving}, i.e.,~$c\sim_s d$ implies $g\cdot c\sim_s g\cdot d$ for all~$s\in S$,~$c,d\in \caC$ and~$g\in G$.
   Then, for every chamber~$c\in \caC$, the following hold:
   \begin{itemize}
       \item[(i)] for every~$s\in S$, the setwise stabiliser~$P_s(c)$ of the $s$-panel~$[c]_s$ acts transitively on that panel. In particular, $P_s(c)$ is a compact open subgroup of~$G$;
       \item[(ii)] $G=\langle P_s(c)\mid s\in S\rangle$. In particular, $G$ is unimodular.
   \end{itemize}
\end{prop}

\begin{proof}
  (i) Let~$g\in G$ and~$s\in S$. Since~$g$ preserves the $s$-adjacencies in~$\caC$, one has~$g^{-1}\cdot [g\cdot c]_s=\{g^{-1}\cdot d\mid d\in \caC,\,g\cdot c\sim_s d\}=[c]_s$, i.e.,
  \begin{equation}\label{eq:panel1}
      [g\cdot c]_s=g\cdot[c]_s.
  \end{equation}
  Since~$G$ acts transitively on~$\caC$, for every~$d\in[c]_s$ there is~$g\in G$ such that~$d=g\cdot c$. By~\eqref{eq:panel1}, necessarily~$g\in P_s(c)$. Hence, $P_s(c)$ acts transitively on~$[c]_s$, which is a finite set by the locally finiteness of the chamber system. By the orbit-stabiliser theorem, $P_s(c)$ contains the stabiliser of~$c$ as a subgroup of finite index. As the latter subgroup is compact open by hypothesis,~$P_s(c)$ is compact open as well.

  (ii) By~(i), for all~$s\in S$ and all chambers~$c,d\in \caC$, 
  \begin{equation*}
      c\sim_s d\,\Longleftrightarrow\, \exists\,g\in P_s(c)\text{ s.th.}~d=g\cdot c.
  \end{equation*}
  This observation and the connectedness of~$\caC$ imply that~$G=\langle P_s(c)\mid s\in S\rangle$. Let now~$\Delta$ and~$\Delta_s$ be the modular functions of~$G$ and of~$P_s(c)$, respectively. It is well-known that~$\Delta\vert_{P_s(c)}=\Delta_s$ since~$P_s(c)$ is an open subgroup of~$G$, and that~$\Delta_s=1$ because~$P_s(c)$ is compact. Therefore, $\Delta\vert_{P_s(c)}=1$ for every~$s\in S$. Since~$\Delta$ is multiplicative, one concludes that~$\Delta=1$ and~$G$ is unimodular.
\end{proof}

\begin{defi}[\protect{\cite[pp.~224-225]{ab:build}}] \label{defi:building} A {\em building}\footnote{This definition of building does not explicitly involve a system of apartments. The existence of such a system can be proved as a consequence of the axioms (see Fact~\ref{fact:Wbuild}). Nevertheless, the given definition is equivalent (see~\cite[Corollary~5.93]{ab:build}) to the original definition of buildings as simplicial complexes by J.~Tits~\cite[p.~173]{ab:build}, which instead involves a system of apartments.} \emph{of type $(W,S)$} is a pair~$(\caC,\delta)$ that consists of a chamber system $\caC$ and a map $\delta\colon\caC\times\caC\to W$, called the {\em Weyl distance} function, and that satisfies the following properties:
\begin{enumerate}
    \item[(B1)] for every~$s\in S$, each $s$-panel contains at least two chambers, and
    \item[(B2)] for any $c,d\in\caC$ and any reduced $S$-word $(s_1,\dots,s_n)$ there exists a gallery of type $(s_1,\dots,s_n)$  connecting $c$ and $d$ if and only if $\delta(c,d)=s_1\dots s_n.$
\end{enumerate}
\end{defi}

An example of a building of type $(W,S)$ is the \textit{standard Coxeter complex $\Sigma(W,S)=(W, \delta_W)$ of type $(W,S)$},
where $\delta_W(w_1, w_2)=w_1^{-1}w_2$ for all $w_1, w_2\in W$ (cf.~\cite[Example~18.1.3]{dav:book}).
Given two buildings $\Delta=(\caC, \delta)$ and $\Delta'=(\caC', \delta')$ of type $(W,S)$, a map $\varphi\colon\caC\to \caC'$ is a \textit{$W$-isometry} if, for all $c_1, c_2\in \caC$,
\begin{equation*}
    \delta'(\varphi(c_1), \varphi(c_2))=\delta(c_1, c_2).
\end{equation*}
Such a map is necessarily injective. Throughout, given a building~$\Delta=(\caC,\delta)$ of type~$(W,S)$, we denote by~$\Aut(\Delta)$ the group of all surjective $W$-isometries~$\caC\to \caC$, and endow it with the pointwise-convergence topology. Notice that~$\Aut(\Delta)$ is also known as the group of all type-preserving automorphisms of~$\Delta$.

Following~\cite[p.~333]{dav:book}, for every building $\Delta=(\caC, \delta)$ of type $(W,S)$ there is a $W$-isometry~$W\to\caC$ mapping~$1_W$ to a prescribed chamber~$c\in \caC$. An \textit{apartment} of~$\Delta$ is any $W$-isometric image of~$W$ in~$\caC$. 
The following fact collects some standard properties about apartments that will be used later on.
\begin{fact}[\protect{cf.~\cite[p.~333]{dav:book}~and~\cite[\S 4]{kramer}}]\label{fact:Wbuild}
    Let $\Delta=(\caC, \delta)$ be a building of type $(W,S)$.
        \begin{itemize}
            \item[(i)] For every apartment $\Sigma$ of $\Delta$, the chamber system $(\caC(\Sigma), \delta\vert_{\caC(\Sigma)})$ is a building of type $(W,S)$. 
          \item[(ii)] There exists a collection of apartments $\mathfrak{A}$, called \emph{atlas} of $\Delta$, with the following property: every two chambers $c$ and $d$ of $\Delta$ are both contained in the set of chambers of some $\Sigma\in \mathfrak{A}$. In particular, one has
            \begin{equation*}
                  \caC=\bigcup_{\Sigma\in \mathfrak{A}}\caC(\Sigma).
            \end{equation*}
        \end{itemize}
\end{fact}

For all $J\subseteq S$ and $c\in \caC$, the \textit{$J$-residue of $\Delta$ centred at $c$} is defined as
\begin{equation*}\label{eq:res}
    \mathrm{Res}_J(c):=\{d\in \caC \mid \delta(c,d)\in W_J\}.
\end{equation*}
In other words, $\Res_J(c)$ is the collection of all $d\in \caC$ for which either $d=c$ or there is a gallery $(c_0=c, c_1,\dots, c_n=d)$ in $\Delta$ whose type $(s_1, \dots, s_n)$ satisfies $s_i\in J$ for every $1\leq i\leq n$ (cf.~\cite[Definition~5.26]{ab:build}). In particular, for all $c,d\in \caC$ and $J\subseteq S$, one deduces that
\begin{equation}\label{eq:res=}
    \Res_J(c)=\Res_J(d) \Longleftrightarrow \delta(c,d)\in W_J.
\end{equation}
A $J$-residue is called \textit{spherical} if $J$ is a spherical subset of $S$. The building $\Delta$ is said to be \textit{locally finite} if, for every $s\in S$, all $\{s\}$-residues are finite. Note that if~$\Delta$ is locally finite, then~$\Aut(\Delta)$ is a t.d.l.c.~group (cf.~\cite[Theorem~5.1]{kramer}).

Denote by $\caR$ the set of all residues of $\Delta$, by $\caR_{\mathrm{sph}}$ the set of all spherical ones and, given $J\subseteq S$, let $\caR_J$ be the set of all $J$-residues of $\Delta$.


\subsection{The Davis complex of a building}\label{sss:dav}
Following \cite[\S 5]{kramer}, we recall here a notable simplicial complex associated to a building $\Delta=(\caC,\delta)$. The \textit{Davis complex $\DeltaDav$ of $\Delta$} is the simplicial complex whose $k$-simplices, for $k\geq 0$, are all chains 
\begin{equation}
    \mathrm{Res}_{J_0}(c)\subsetneq \dotsm \subsetneq  \mathrm{Res}_{J_k}(c),
\end{equation}
where $c\in \caC$ and $\mathrm{Res}_{J_0}(c), \dots, \mathrm{Res}_{J_k}(c)\in \caR_{\mathrm{sph}}$ with $J_0\subsetneq J_1 \subsetneq \dotsm \subsetneq J_k$.
The simplicial complex $\DeltaDav$ is locally finite if and only if $\Delta$ is locally finite (cf.~\cite[p.~22]{kramer}). By~\cite[Corollary~18.3.6]{dav:book}, the geometric realisation $|\DeltaDav|$ of $\DeltaDav$ is contractible (with respect to the weak topology) and is usually called the \textit{Davis realisation of $\Delta$}. Moreover, since $|\DeltaDav|$ is contractible, the augmented cellular chain complex associated to it is exact (cf.~\cite[\S A.3]{cw:qrat} and~\cite[p.~14]{br:coho}).
\subsection{Cayley--Abels graphs}\label{ss:cayab} 
Given a t.d.l.c.~group~$G$, a {\em Cayley--Abels graph of $G$} is defined to be a connected locally finite undirected graph $\Gamma$ on which $G$ acts vertex-transitively and with compact open vertex stabilisers. By~\cite[Theorem~2.2]{km:cayley}, $G$ admits a Cayley--Abels graph if and only if $G$ is compactly generated. Moreover, all Cayley--Abels graphs of a compactly generated t.d.l.c.~group $G$ are quasi-isometric to each other (cf.~\cite[Theorem~2.7]{km:cayley}).
\begin{example}\label{ex:cayab}
    Let~$\Delta=(\caC,\delta)$ be a locally finite building of type $(W,S)$. If $G$ acts chamber-transitively on $\Delta$ with compact open stabilisers, then the \emph{chamber graph $(\caC, E(\caC))$ of $\Delta$} (cf.~Remark~\ref{rem:chamgr} or \cite[Definition~4.1]{kramer}) is a Cayley--Abels graph of~$G$. 
\end{example}
Following~\cite[p.~366]{cast:cd1}, a {\em compact generating system} of $G$ is a pair  $(K,S)$ such that
\begin{itemize}
    \item $K$ is a compact open subgroup of $G$;
    \item $S$ is a finite subset of $G$ satisfying $S=S^{-1}$ and $S\cap K=\emptyset$;
    \item $G$ is algebraically generated by $K\cup S$.
\end{itemize} 
Given a compact generating system $(K,S)$ of $G$, the {\em associated Cayley--Abels graph $\Gamma (G,K,S)$} is the connected graph  consisting of the following data:
$$V(\Gamma (G,K,S)) = G/K\quad\text{and}\quad E(\Gamma (G,K,S)) = \{ \{gK, gsK\} \mid g \in G, s \in S \},$$
where $G/K$ denotes the
set of left cosets of $K$ in $G$.
\subsection{Rational discrete cohomology for t.d.l.c.~groups}\label{ss:coh}
In this section, we collect some background on the rational discrete cohomology for t.d.l.c.~groups that will be used later. For further details the reader is referred to~\cite{cw:qrat}.

The category~${}_{\QG}\mathbf{dis}$ of \emph{discrete left $\Q[G]$-modules} is the full subcategory of the category of left $\QG$-modules the objects of which are all $\Q$-vector spaces~$M$ equipped with a discrete left $G$-action, meaning that the stabiliser of every element in $M$ is an open subgroup of $G$. The category ${}_{\QG}\mathbf{dis}$ is an abelian category with enough injectives (cf.~\cite[Section~2.2]{cw:qrat}).
This allows one to define the right derived functor 
$$\mathrm{dExt}^n_G(M,-): {}_{\QG}{\mathbf{dis}} \to {}_{\Q}{\mathbf{ mod}}$$
of~$\Hom_{\QG}(M,-)$ for every object $M$ and~$n\geq 0$. Subsequently, the \emph{rational discrete cohomology} of $G$ is defined as
$$\dH^n(G,-):=\mathrm{dExt}^n_G(\Q,-) \quad \forall\,n\geq 0.$$

By~\cite[Proposition~3.2]{cw:qrat}, ${}_{\QG}\mathbf{dis}$ also has enough projectives. Hence, $\dH^n(G,-)$ can be computed using any projective resolution $P_\bullet\to\Q$ in ${}_{\QG}\mathbf{dis}$ of the trivial module~$\Q$ (see for example~\cite[Theorem~2.7.6]{weibel}). Note that, if the t.d.l.c.~group $G$ acts on a contractible CW-complex $X$ with compact open vertex stabilisers, the projective resolution $P_\bullet\to\Q$ can be taken to be the $\Q$-tensored augmented cellular chain complex associated with $X$ (this is a common strategy, whose details can be found in~\cite[\S 5.5]{ccw:zeta} for example). This strategy will often be adopted for t.d.l.c.~groups that act on locally finite buildings with compact open stabilisers by taking the Davis complex as the contractible CW-complex (cf.~Section~\ref{sss:dav}). The reason why the augmented cellular chain complex produces a projective resolution of~$\Q$ in~${}_{\QG}\mathbf{dis}$ is~\cite[Corollary 3.3]{cw:qrat}: the projective objects in ${}_{\QG}\mathbf{dis}$ are exactly the direct summands of the so-called \emph{proper discrete permutation modules}, i.e.,
$$\bigoplus_{i\in I} \Q[G/U_i],$$
where~$I$ is a non-empty set and ~$U_i$ is a compact open subgroup of~$G$ for all $i\in I$. 

For~$n\geq 0$, a t.d.l.c.~$G$ is {\em of type~$\mathrm{FP}_n$ over $\Q$} if the trivial ${}_{\QG}\mathbf{dis}$-module~$\Q$ admits a resolution $P_{\bullet}\twoheadrightarrow \Q$ by proper discrete permutation modules 
  such that $P_k$ is finitely generated for all $k \leq n$.  Note that every~$G$ is of type~$\mathrm{FP}_0$ over $\Q$. If~$G$ is of type~$\mathrm{FP}_n$ over $\Q$ for every~$n$, then~$G$ is said to be of type~$\mathrm{FP}_\infty$ over~$\Q$.

Finally, for a 
t.d.l.c.~group $G$ we denote by $\ccd_\Q(G)$ the \emph{rational discrete cohomological dimension} of $G$ as defined in~\cite[\S 3.4]{cw:qrat}. It can be characterised as $$\ccd_\Q(G)=\sup\big\{k\in \Z_{\geq 0}\mid \dH^k(G,\argu)\neq 0\big\},$$
cf.~\cite[Equation~(3.11)]{cw:qrat}.
\begin{fact}[\protect{\cite[Example~2.6, Proposition 3.7(a)]{cw:qrat}}]\label{fact:discoho}
Let $G$ be a t.d.l.c.~group.
\begin{enumerate}
\item $\QG$ is an object of ${}_{\QG}\mathbf{dis}$ if and only if $G$ is discrete.
\item  If $G$ is a discrete group, then 
$\dH^\bullet(G,-)$ coincides with the ordinary group cohomology $\mathrm H^\bullet(G,- )$.
    \item $G$ is compact if and only if $\ccd_\Q(G)=0$.
\end{enumerate}
\end{fact}


\section{The number of ends of a compactly generated t.d.l.c.~group}
\subsection{Almost $(G,K)$-invariant sets} 
Let $K$ be a compact open subgroup of~$G$. Following \cite[\S 3.3]{cast:cd1} and~\cite[p.~196-197]{dun:acc}, a set $B \subseteq G/K$ is said to be
{\em almost $(G, K)$-invariant} if, for all $g \in G$, $gB =_a B$ (i.e., the symmetric difference of~$gB$ and~$B$ is finite) and, for all $k \in K$, $k B = B$.

For every left $K$-invariant set~$B\subseteq G/K$, define
\begin{equation}\label{eq:CB}
    C_B:=\{gK\in G/K\mid g^{-1}K\in B\}.
\end{equation}
Note that the definition of~$C_B$ already appears in~\cite[proof of Lemma~4.5]{cast:cd1}.
Notice that
\begin{equation}\label{eq:CBfinite}
    |B|<\infty\,\Longleftrightarrow\,|C_B|<\infty.
\end{equation}
Indeed, let $\pi\colon G\to G/K$ be the canonical projection. Since~$K$ is compact, $B$ is finite if and only if $\tilde{B}=\pi^{-1}(B)$ is compact. Analogously, $C_B$ is finite if and only if  $\pi^{-1}(C_B)=\tilde{B}^{-1}=\{\,b^{-1}\mid b\in\tilde{B}\,\}$ is compact. Thus, $B$ is finite if and only if $C_B$ is finite.

Note also that if $G$ is unimodular, the same argument also shows that~$|B|=|C_B|$.
\subsection{Almost $(G,K)$-invariant functions}
Let~$(K, S)$ be a compact generating system of $G$ (see Section~\ref{ss:cayab}). In this section we recall the relation between the rational discrete first-degree cohomology of $G$ and the almost $(G,K)$-invariant sets. The $\Q$-vector space~$\mathcal F(G/K, \Q)$ of all functions from~$G/K$ to~$\Q$ admits a left $G$-action
given by
\[g \cdot \alpha(x) = \alpha(g^{-1}
x),\quad \alpha\in\mathcal F(G/K, \Q), g \in G,  x \in G/K.\]
One writes $\alpha =_a \beta$  if $\alpha(x) = \beta(x)$ for all but finitely many elements $x \in G/K$.
An element $\alpha \in \mathcal F(G/K, \Q)$ is said to be {\em almost $(G, K)$-invariant} if, for all $g \in G$,
$g \cdot \alpha =_a \alpha$  and, for all $k \in K$,  $k \cdot \alpha = \alpha$. For example, the characteristic function $\chi_B$ of an almost $(G,K)$-invariant set $B$ is an almost $(G,K)$-invariant function.  By~\cite[Proposition~3.15]{cast:cd1},
for every compact open subgroup $K$ of a t.d.l.c.~group $G$ one has
\begin{equation}\label{eq:alm inv}
\dH^1(G, \Q[G/K])\cong\frac{\mathrm{AInv}_K (G, \Q)}{C(G/K) + \Q[G/K]^K},
\end{equation}
where $\mathrm{AInv}_K (G, \Q)$ denotes the $\Q$-vector
space of all almost $(G, K)$-invariant functions, 
$C(G/K)$ denotes the $\Q$-vector space of constant functions from~$G/K$ to~$\Q$, and $\Q[G/K]^K$ is the largest $K$-invariant submodule of $\Q[G/K]$. Here $\Q[G/K]^K$ is regarded as the space of all almost zero functions in $\mathrm{AInv}_K(G, \Q)$. In view of the proof of Lemma~\ref{lem:ends1} it is worth recalling that the $\Q$-vector space  $\mathrm{AInv}_K (G, \Q)$ is {spanned} by the set of characteristic functions
\begin{equation}\label{eq:AInvSpan}
    \{\chi_B\mid B\subset G/K\ \text{almost $(G,K)$-invariant set}\},
\end{equation}
cf.~\cite[Lemma~4.4]{cast:cd1}. 

\subsection{Number of ends}\label{ss:ends}
A locally finite connected {undirected} graph~$\Gamma$ has at least $n$~ends if there is a finite subgraph~$\Gamma'$ such that~$\Gamma-\Gamma'$ has at least~$n$ infinite connected components. The graph~$\Gamma$ has~$n$ ends if it has at least $n$ ends but not at least $n+1$ ends. Moreover, $\Gamma$ has infinitely many ends if, for all $n\in\Z_{\geq 0}$, $\Gamma$ has at least $n$ ends. The {\em number of ends $e(G)$} of $G$ is the number of ends of a Cayley--Abels graph of $G$ (cf.~Section~\ref{ss:cayab}). Note that this number is well-defined since the Cayley--Abels graphs of $G$ are quasi-isometric to each other. It is well-known that $e(G)\in\{0,1,2,\infty\}$ and that $e(G)=0$ if and only if $G$ is compact (see~\cite[Theorem~3.4]{km:cayley} and the references therein).

The following results, and the techniques to prove them, take inspiration from the proof of~\cite[Satz~IV]{specker}.

\begin{lem}\label{lem:ends1}
      Let $G$ be a non-compact compactly generated t.d.l.c.~group and let~$K$ be a compact open subgroup of~$G$. Then
      \begin{equation*}
          \dim_\Q\dH^1(G,\Q[G/K])\leq e(G)-1.
      \end{equation*}
     
\end{lem}
\begin{proof}
    Let~$m\in \Z_{\geq 1}$ and assume that the dimension of $\dH^1(G, \Q[G/K])$ is at least~$m$.
By \eqref{eq:alm inv} and~\eqref{eq:AInvSpan}, there exist
$m$ almost $(G,K)$-invariant subsets~$B_1,\dots, B_m$ of~$G/K$ whose characteristic functions $\chi_{B_1}, \dots, \chi_{B_m}$ are linearly independent modulo $C(G/K)+\Q[G/K]^K$.
Denote by $\delta C_{B_i}$ the set of vertices~$gK$ in $\Gamma(G,K,S)$ that are not in $C_{B_i}$ but are adjacent to a vertex of $C_{B_i}$ (i.e., $gK\not\in C_{B_i}$ but~$gksK\in C_{B_i}$ for some~$k\in K$ and~$s\in S$). Since~$B_i$ is~$K$-invariant, notice that
\begin{equation*}
    \delta C_{B_i}=\Big\{gK\,\Big \vert\, g\in G,\,g^{-1}K\in \textstyle{\bigcup_{\substack{s\in S\\ k\in K}}}k(sB_i\setminus B_i)\Big\}.
\end{equation*}
As $B_i$ is infinite and almost $(G,K)$-invariant,~$C_{B_i}$ is infinite (cf.~\eqref{eq:CBfinite}). 
Moreover, for every~$s\in S$, the compact group~$K$ acts on the set~$\bigcup_{k\in K}k(sB_i\setminus B_i)$ with open stabilisers. By the orbit-stabiliser theorem, for every~$xK\in sB_i\setminus B_i$, the $K$-orbit~$K\cdot xK$ has order~$|K:K\cap xKx^{-1}|<\infty$. Since both~$sB_i\setminus B_i$ and~$S$ are finite, also $\bigcup_{\substack{s\in S\\k\in K}}k(sB_i\setminus B_i)$ is finite. This proves that~$\delta C_{B_i}$ is finite.

Let $\Gamma'$ be the subgraph of $\Gamma(G,K,S)$ spanned by $\delta C_{B_1}\cup\dots\cup\delta C_{B_m}$. 
Since the graph~$\Gamma(G,K,S)$ is connected and locally finite and~$\Gamma'$ is finite,~$\Gamma(G,K,S)-\Gamma'$ has only finitely many connected components. Denote by~$D_1,\dots,D_p$, where~$p\in \Z_{\geq 0}$, the vertex sets of all the infinite connected components of~$\Gamma(G,K,S)-\Gamma'$, respectively.
Since~$\Gamma(G,K,S)$ is infinite and~$\Gamma'$ is finite, one has~$p\geq 1$. Moreover,~$p\leq e(G)$. Below we show that~$p-1\geq m$.

Observe that 
\begin{equation*}
    G/K=V(\Gamma(G,K,S))=F\sqcup \bigsqcup_{j=1}^pD_j,
\end{equation*}
where~$F$ is the finite set that contains~$\delta C_{B_1}\cup \dots \cup \delta C_{B_m}$ and all vertices of all the (finitely many) finite connected components of $\Gamma(G,K,S)-\Gamma'$.
In particular, for every $1\leq i\leq m$ one has
\begin{equation}\label{eq:CBi}
    C_{B_i}=(C_{B_i}\cap F)\sqcup \bigsqcup_{j=1}^p(C_{B_i}\cap D_j).
\end{equation}
We claim that, for all $1\leq i\leq m$ and $1\leq j\leq p$, either $C_{B_i}\supseteq D_j$ or $C_{B_i}\cap D_j=\emptyset$. Suppose indeed that 
there is $gK\in D_j\setminus C_{B_i}$. Since $D_j\cap \delta C_{B_i}=\emptyset$, {each vertex adjacent to some} $hK\in D_j\setminus C_{B_i}$ does not belong to $C_{B_i}$. Hence, every path in $\Gamma(G,K,S)$ with vertices in $D_j$ starting at~$gK$ has all its vertices in $D_j\setminus C_{B_i}$. Since $D_j$ spans a connected subgraph of $\Gamma(G,K,S)$, every vertex in $D_j$ can be connected to $gK$ by such a path. Then $D_j=D_j\setminus C_{B_i}$, which yields $D_j\cap C_{B_i}=\emptyset$.
Therefore, by \eqref{eq:CBi}, for every $1\leq i\leq m$ there is $\{j_1, \dots, j_{r_i}\}\subseteq \{1, \dots,p\}$ such that 
\begin{equation}\label{eq:chiCBi}
    \chi_{C_{B_i}}=\chi_{C_{B_i}\cap F}+\sum_{k=1}^{r_i}\chi_{D_{j_k}}.
\end{equation}
Denote by $\mathcal{A}\caC(G/K,\Q)$ the space of all functions $f\colon G/K\longrightarrow \Q$ which are {\em constant almost everywhere}, i.e., $f$ is constant on all but finitely many elements of~$G/K$. Moreover,
 $\mathcal{A}\caC(G/K,\Q)$ is regarded as a subspace of the $\Q$-vector space~$\mathcal{F}(G/K,\Q)$ of all maps from~$G/K$ to~$\Q$. Since $\chi_{C_{B_i}\cap F}\in \mathcal{A}\caC(G/K,\Q)$ for every $1\leq i\leq m$, \eqref{eq:chiCBi} implies that the $\Q$-subspace $\caS_1$ in $\mathcal{F}(G/K,\Q)/\mathcal{A}\caC(G/K,\Q)$ generated by 
$\{\chi_{C_{B_i}}+\mathcal{A}\caC(G/K,\Q)\}_{i=1}^m$ is contained in the $\Q$-subspace $\caS_2$ generated by $\{\chi_{D_j}+\mathcal{A}\caC(G/K,\Q)\}_{j=1}^p$. Note that 
$$\dim_\Q \caS_1=m.$$ 
Indeed, $\chi_{B_i}, \dots, \chi_{B_m}$ are linearly independent modulo $C(G/K)+\Q[G/K]^K$ (and so modulo $\mathcal{A}\caC(G/K,\Q)$) and, for all $t_1,\dots, t_m\in \Q$ and $x\in G$, one has
\begin{equation*}
    \Big(t_1\chi_{C_{B_1}}+\dots +t_m\chi_{C_{B_m}}\Big)(xK)=\Big(t_1\chi_{B_1}+\dots +t_m\chi_{B_m}\Big)(x^{-1}K).
\end{equation*}
Moreover, {we claim that}
\begin{equation}\label{eq:dimS2}
    \dim_\Q \caS_2=p-1.
\end{equation}
Firstly, notice that~$\chi_{D_1}, \dots, \chi_{D_p}$ are linearly dependent modulo $\mathcal{A}\caC(G/K,\Q)$ by observing, for example, that $\sum_{j=1}^p\chi_{D_j}$ is the constant function~$1$ almost everywhere.
However, provided $p\geq 2$,~$\chi_{D_1},\dots,\chi_{D_{p-1}}$ are linearly independent modulo $\mathcal{A}\caC(G/K,\Q)$. Indeed, every $\Q$-linear combination
\begin{equation}\label{eq:lincomchi}
    t_1\chi_{D_1}+\dots +t_{p-1}\chi_{D_{p-1}}
\end{equation}
vanishes at least on the infinite set~$D_p$. Hence, the only $\Q$-linear combination as in~\eqref{eq:lincomchi} that vanishes almost everywhere on~$G/K$ is the one with~$t_1=\dots=t_{p-1}=0$. This proves~\eqref{eq:dimS2}, and one can now conclude that
\begin{equation*}
    e(G)-1\geq \dim_\Q\caS_2\geq \dim_\Q \caS_1=m.\qedhere
\end{equation*}
\end{proof}
\begin{thm}\label{thm:ends}
 Let $G$ be a non-compact compactly generated t.d.l.c.~group and let~$K$ be a compact open subgroup of~$G$. 
    Then there is an open subgroup~$U\leq K$ such that, for every open subgroup~$\caO\leq U$, one has
  \begin{equation*}
      \dim_\Q \dH^1(G,\Q[G/\caO])=e(G)-1.
  \end{equation*} 
\end{thm}
\begin{proof}
  Consider a compact generating system~$(K,S)$ of~$G$ (see Section~\ref{ss:cayab}) and let~$n\in \Z_{\geq 1}$ be such that~$\Gamma=\Gamma(G,K,S)$ has at least~$n$ ends.
  Hence, there is a finite subgraph~$\Gamma'$ of~$\Gamma$ such that~$\Gamma-\Gamma'$ has at least~$n$ infinite connected components whose vertex sets are denoted by~$C_1,\dots, C_n$, respectively. For every~$1\leq i\leq n$, choose~$v_i\in C_i$ and let~$U$ be the pointwise stabiliser in~$K$ of~$V(\Gamma')\cup\{v_1,\dots, v_n\}$. Note that~$U$ is a finite intersection of conjugates of~$K$ in~$G$. Therefore, $U$ is an open subgroup of~$K$ and hence a compact open subgroup of~$G$.

   For every~$1\leq i\leq n$, observe that
   \begin{equation}\label{eq:endthm1}
       uC_i=C_i,\qquad\forall\,u\in U.
   \end{equation}
   Indeed, since~$U$ fixes~$V(\Gamma')$ pointwise, we have~$uC_i\subseteq V(\Gamma-\Gamma')$. Moreover, because~$G$ acts on~$\Gamma$ by simplicial automorphisms,~$uC_i$ spans a connected subgraph of~$\Gamma$ that contains~$u\cdot v_i=v_i$. Hence, $uC_i$ is contained in the unique connected component of~$\Gamma-\Gamma'$ containing~$v_i$, i.e.,~$uC_i\subseteq C_i$. This holds for every~$u\in U$ and, since~$U$ is closed under inversion, this yields~\eqref{eq:endthm1}.

   Let now~$\caO\leq U$ be an arbitrary open subgroup. Note that~$\caO$ is a compact open subgroup of~$G$ and that, for all~$\omega\in \caO$ and~$1\leq i\leq n$, one has~$\omega C_i=C_i$ by~\eqref{eq:endthm1}. 
    
 For every~$1\leq i\leq n$, define
  \begin{equation}
      D_i:=\{g\caO\mid g\in G,\,gK\in C_i\}\subseteq G/\caO.
  \end{equation}
  Note that~$D_i$ is well-defined as~$\caO\leq K$.
  Since~$C_i$ is $\caO$-invariant, so is~$D_i$. In particular, one may define~$C_{D_i}$ as in~\eqref{eq:CB}. Notice that
  \begin{equation}\label{eq:BDi}
      C_{D_i}=\{g\caO \mid g^{-1}K\in C_i\}.
  \end{equation}
We claim that
\begin{itemize}
    \item[(i)] $C_{D_i}$ is almost~$(G,\caO)$-invariant for every $i=1,\dots,n$;
    \item[(ii)] the indicator functions $\chi_{C_{D_1}},\dots, \chi_{C_{D_{n-1}}}\in\mathrm{AInv}_\caO(G,\Q)$ are linearly independent modulo~$C(G/\caO)+\Q[G/\caO]^\caO$.
\end{itemize}
We postpone the proof of items~(i)-(ii), and now argue how to use them to conclude the argument. 
   In detail, items~(i)-(ii) together with~\eqref{eq:alm inv} imply that
   \begin{equation}\label{eq:endthm11}
       \dim_\Q\dH^1(G,\Q[G/\caO])\geq n-1.
   \end{equation}
   From~\eqref{eq:endthm11}, the statement follows easily.
   Indeed, if~$e(G)=\infty$ then~$\Gamma(G,K,S)$ has at least~$n$ ends for every~$n\geq 1$, and~\eqref{eq:endthm11} yields~$\dim_\Q\dH^1(G,\Q[G/\caO])=\infty=e(G)-1$. By Lemma~\ref{lem:ends1}, in case
  that~$e(G)<\infty$, the inequality~\eqref{eq:endthm11} with~$n=e(G)$ becomes an equality.
   
   It remains to prove the claims~(i) and~(ii).

  (i) $C_{D_i}$ is left $\caO$-invariant because it is $K$-invariant by definition (cf.~\eqref{eq:BDi}). For the almost $G$-invariance, since~$G=\langle S\cup K\rangle$, it is sufficient to prove that~$s C_{D_i}=_a C_{D_i}$ for every~$s\in S$. As $S=S^{-1}$,
it suffices to show that $$C_{D_i}\setminus sC_{D_i}=\{g\caO\in C_{D_i}\mid s^{-1}g\caO\notin C_{D_i}\}$$ 
  is finite for every~$s\in S$. Let~$g\caO\in C_{D_i}$ be such that~$s^{-1}g\caO\not\in C_{D_i}$. By~\eqref{eq:BDi}, this means that
  \begin{equation*}
      g^{-1}K\in C_i\text{ and }g^{-1}sK\not\in C_i.
  \end{equation*}
  In other words, $g^{-1}sK$ is in the \emph{boundary~$\delta C_i$} of~$C_i$ in~$\Gamma=\Gamma(G,K,S)$, i.e., $g^{-1}sK$ is not a vertex of~$C_i$ but it is connected to a vertex of~$C_i$ via an edge of~$\Gamma$ (which, in this case, is~$\{g^{-1}K,g^{-1}sK\}$). Notice that~$\delta C_i$ cannot contain any vertex of~$\Gamma-\Gamma'$, 
  otherwise there would be an edge connecting a vertex in~$V(\Gamma-\Gamma')\setminus C_i$ and a vertex in~$C_i$, against the fact that~$C_i$ spans a connected component of~$\Gamma-\Gamma'$. Hence, $\delta C_i\subseteq V(\Gamma')$. We conclude that
  \begin{equation}\label{eq:endthm2}
      C_{D_i}\setminus sC_{D_i}\subseteq \{g\caO \mid g\in G,\,g^{-1}sK\in V(\Gamma')\}.
  \end{equation}
 Set~$V(\Gamma')=\{g_1 K,\dots, g_mK\}$ for some~$m\geq 1$ and $g_1,\dots,g_m\in G$, and observe that
  \begin{equation*}
      g^{-1}sK\in V(\Gamma')\,\Longleftrightarrow\,g\in \bigcup_{\ell=1}^m sKg_\ell^{-1}.
  \end{equation*}
  Hence, denoting by~$\pi_\caO\colon G\twoheadrightarrow G/\caO$ the canonical projection, one has
  \begin{equation}\label{eq:endthm3}
      \{g\caO \mid g\in G,\,g^{-1}sK\in V(\Gamma')\}=\pi_\caO\Big(\textstyle{\bigcup_{\ell=1}^m} sKg_\ell^{-1}\Big).
  \end{equation}
  But~$\textstyle{\bigcup_{\ell=1}^m} sKg_\ell^{-1}$ is compact in~$G$ (because so is~$K$) and~$\pi_\caO$ is a continuous map onto the discrete space~$G/\caO$. Hence, by~\eqref{eq:endthm3}, the set~$\{g\caO \mid g\in G,\,g^{-1}sK\in V(\Gamma')\}$ is finite. In turn, by~\eqref{eq:endthm2}
 one deduces that~$C_{D_i}\setminus sC_{D_i}$ is finite.
 This proves that~$C_{D_i}$ is almost~$(G,\caO)$-invariant for every~$1\leq i\leq n$. 
 
 (ii) By~(i), the functions $\chi_{C_{D_1}},\dots, \chi_{C_{D_n}}$ belong to~$\mathrm{AInv}_\caO(G,\Q)$. We now check that~$\chi_{C_{D_1}},\dots, \chi_{C_{D_{n-1}}}$ are $\Q$-linearly independent modulo~$C(G/\caO)+\Q[G/\caO]^\caO$. Notice that each~$D_i$ is infinite, as it is the preimage of the infinite set~$C_i$ via the canonical projection~$G/\caO\twoheadrightarrow G/K$. By~\eqref{eq:CBfinite}, each~$C_{D_i}$ is infinite as well. With a similar argument, for every~$1\leq i,j\leq n$ with~$i\neq j$, since~$C_i\cap C_j=\emptyset$ one gets that~$D_i\cap D_j=\emptyset$ and hence~$C_{D_i}\cap C_{D_j}=\emptyset$.
Therefore, every~$\Q$-linear combination
 \begin{equation}\label{eq:endthm4}
     t_1\chi_{C_{D_1}}+\dots +t_{n-1}\chi_{C_{D_{n-1}}}
 \end{equation}
vanishes on the infinite set~$C_{D_{n}}$. In particular, the only $\Q$-linear combination as in~\eqref{eq:endthm4} that is constant almost everywhere on~$G/\caO$ is the one with~$t_1=\dots =t_{n-1}=0$. This concludes the argument.  
 \end{proof}

Let $\mathcal{C}\mathcal{O}(G)$ be the poset of all compact open subgroups of $G$ ordered by reverse inclusion $\preceq$. Recall that the \emph{rational discrete standard bimodule of~$G$} is defined in~\cite[Section~4.2]{cw:qrat} as 
\begin{equation}\label{eq:biG}
    \BiG:=\varinjlim_{K\in \mathcal{C}\mathcal{O}(G)}\Big(\Q[G/K],\eta_{K,H}\Big),
\end{equation}
where, for every $K\prec H$, $\eta_{K,H}\colon\Q[G/K]\to\Q[G/H]$ is the $\Q[G]$-module homomorphism given by $\eta_{K,H}(K)={|K:H|}^{-1}\sum_{r\in\caR_{K,H}}rH$. Here $\caR_{K,H}$ denotes a set of representatives of the left cosets of $H$ in $K$ that contains $1$.

\begin{thm}\label{thm:end2}
    Let $G$ be a non-compact compactly generated t.d.l.c.~group. Then
    $$\dim_\Q\dH^1(G,\BiG)=e(G)-1.$$
\end{thm}
\begin{proof}
    For all compact open subgroups $H\subseteq K$ of $G$, 
$\eta_{K,H}\colon\Q[G/K]\to\Q[G/H]$ induces the injective map $\dH^1(\eta_{K,H})\colon \dH^1(G,\Q[G/K])\to \dH^1(G,\Q[G/H])$ (cf.~\cite[End of proof of Proposition~4.7]{cw:qrat}).
    By~\cite[Remark~3.4]{cast:cd1}, one deduces that
$$\dH^1(G,\BiG)\cong\varinjlim_{K\in\mathcal{CO}(G)}\big(\dH^1(G,\Q[G/K]),\dH^1(\eta_{K,H})\big),$$
where the set of all compact open subgroups $\mathcal{CO}(G)$ is ordered by reverse inclusion~$\preceq$.
By Theorem~\ref{thm:ends}, there is~$U\in \mathcal{CO}(G)$ such that, for every open subgroup~$\caO\leq U$, one has
\begin{equation}\label{eq:endformula}
    \dim_\Q \dH^1(G,\Q[G/\caO])=e(G)-1.
\end{equation}
Since~$\{\caO\in \mathcal{CO}(G)\mid \caO\subseteq U\}$ is a cofinal subsystem of the direct system~$(\mathcal{CO}(G),\preceq)$, one has
\begin{equation*}
    \dH^1(G,\BiG)\cong\varinjlim_{K\in\mathcal{CO}(G)}\dH^1(G,\Q[G/K])\cong\varinjlim_{\caO\in\mathcal{CO}(G),\atop \caO\leq U}\dH^1(G,\Q[G/\caO]).
\end{equation*}

    The claim is then implied by the following two facts:

    \smallskip
\noindent\textbf{Fact~1:} If $\dH^1(G,\Q[G/\caO])$ has infinite dimension over $\Q$ for some $\caO\in \mathcal{CO}(G)$, then $\dH^1(G,\BiG)$ 
has infinite dimension over $\Q$
(cf.~\cite[Proposition~1.2.4(c)]{ribzal}).

     \noindent\textbf{Fact~2:} If $\dim_{\Q}\dH^1(G,\Q[G/\caO])=d<\infty$ for every~$\caO\in \mathcal{CO}(G)$ with~$\caO\leq U$, then the space $\dH^1(G,\BiG)$ 
     has dimension~$d$. (To see this it suffices to observe that, by~\eqref{eq:endformula}, for all~$H,K\in\mathcal{CO}(G)$ with~$U\supseteq K\supseteq H$ the linear map~$\dH^1(\eta_{K,H})$ is an isomorphism, and then apply~\cite[Proposition~1.2.4(c)-(d)]{ribzal}.)
\end{proof}

\begin{cor}\label{cor:end3}
    Let $G$ be a compactly generated t.d.l.c.~group. Then, 
    \begin{equation*}
        e(G)=1-\dim_\Q\dH^0(G,\BiG)+\dim_\Q\dH^1(G,\BiG).
    \end{equation*}
\end{cor}

\begin{proof}
    Recall from~\cite[Proposition~4.3(b)]{cw:qrat} that
    \begin{equation}\label{eq:dH0}
   \dH^0(G,\BiG)\cong\mathrm{Hom}_G(\Q,\BiG)\cong \begin{cases}
        \Q,& \text{if~$G$ is compact,}\hfill\\
        \, 0,&\text{otherwise}.\hfill
    \end{cases}
    \end{equation}
    The statement is now a direct consequence of Theorem~\ref{thm:end2} and the fact that for $G$ compact one has~$e(G)=0=\ccd_\Q(G)$ (see Section~\ref{ss:ends} and Fact~\ref{fact:discoho}).
\end{proof}
\section{Invariants of t.d.l.c.~groups acting on buildings}
\subsection{On the rational discrete cohomological dimension}\label{s:coho}

Given a locally finite building $\Delta$, denote by $\Hc^\bullet(|\DeltaDav|, \Q)\cong\Hc^\bullet(|\DeltaDav|, \Z)\otimes_{\Z} \Q$ the cohomology with compact support of the geometric realisation of~$\DeltaDav$ with rational coefficients (cf.~\cite[\S~A.3]{cw:qrat}). Furthermore, let~$C_c(G, \Q)$ be the space of all locally constant functions from~$G$ to~$\Q$ with compact support and let~$\ccd_\Q(W)$ be the ordinary rational cohomological dimension of~$W$ (cf.~Equation~\eqref{eq:cd(W)}). 
\begin{lem}\label{lem:cd}
For every locally finite building $\Delta$ of type~$(W,S)$, one has 
$$\ccd_\Q(W)=\max\{n\geq 0 : \Hc^n(|\DeltaDav|, \Q)\neq 0\}.$$
\end{lem}
\begin{proof}
    The argument used to prove the equality~(6.26) 
    in \cite{cw:qrat} can be transferred verbatim.
\end{proof}
\begin{lem}\label{lem:coho}
    Let~$G$ be a t.d.l.c.~group acting on a locally finite building~$\Delta$ of type~$(W,S)$ with compact stabilisers and finitely many orbits. Then, for every~$k\geq 0$, one has a canonical isomorphism of $\Q$-vector spaces
    \begin{equation}\label{eq:coho1}
          \dH^k(G, C_c(G, \Q))\cong \Hc^k(|\Delta_{\mathrm{Dav}}|, \Q).
    \end{equation}
\end{lem}
\begin{proof}
    The simplicial complex $\Delta_{\mathrm{Dav}}$ is finite-dimensional, locally finite and has contractible geometric realisation (cf.~Section \ref{sss:dav}). By hypothesis, $G$ acts on $\DeltaDav$ with compact open vertex stabilisers (cf.~\cite[Lemma~5.13]{kramer}). Moreover, for all $g\in G$, $c\in \caC$ and $J\subseteq S$, one has $\Res_J(g\cdot c)=g\cdot \Res_J(c)$ and then $G$ has finitely many orbits on each skeleton of $\DeltaDav$. Now the argument used to prove \cite[Equation~(6.13)]{cw:qrat} can be transferred verbatim.
    \end{proof}
\begin{rem}\label{rem:actions}
{More generally, there exist isomorphisms analogous to \eqref{eq:coho1} by replacing~$|\DeltaDav|$ with any locally finite contractible $G$-CW-complex~$X$ satisfying the following properties: 
$G$ acts by cell-permuting homeomorphisms, it has compact open stabilisers and finitely many $G$-orbits in each dimension, and any element in $G$ fixing a cell setwise fixes it pointwise (cf.~\cite[\S 2.1]{ccc} and~\cite[Proof of Theorem~6.7 and~Equation~(6.13)]{cw:qrat} for more details).}
\end{rem}
\begin{thm}\label{thm:ihbuil} 
     Let $G$ be a t.d.l.c.~group acting on a locally finite building~$\Delta$ of type~$(W,S)$ with compact stabilisers and finitely many orbits.
    Then
	\begin{equation}\label{eq:cdGW1}
		\ccd_\Q(G)=\ccd_\Q(W).
	\end{equation}
 In particular, $G$ is compact if and only if $W$ is spherical.
\end{thm}
\begin{proof}  
    Arguing as in the proof of Lemma~\ref{lem:coho}, $G$ acts on $\DeltaDav$ with compact open stabilisers and finitely many orbits on each skeleton. In particular,  $\ccd_\Q(G)\leq \dim(\DeltaDav)$ and $G$ is of type $\mathrm{FP}_\infty$ (cf.~\cite[Proposition~6.6]{cw:qrat} and Section~\ref{ss:coh}). 
   Recall from~\cite[\S 4.7]{cw:qrat} (see in particular~\cite[Equation~(4.64)]{cw:qrat}) that there is a (non-canonical) isomorphism of discrete $\QG$-modules~$\BiG\to C_c(G,\Q)$. 
   Hence,
   $$\ccd_\Q(G)=\max\{n\geq 0 \mid \dH^n(G, \BiG)\neq 0\}=\max\{n\geq 0 \mid \dH^n(G, C_c(G, \Q))\neq 0\},$$
   where the first equality is stated in~\cite[Proposition~4.7]{cw:qrat}.
   Therefore, by Lemma~\ref{lem:coho}, 
   $$\ccd_\Q(G)=\max\{n\geq 0 \mid \Hc^n(|\Delta_{\mathrm{Dav}}|, \Q)\neq 0\}$$
    and hence, by Lemma~\ref{lem:cd}, $\ccd_\Q(G)=\ccd_\Q(W).$
    
    The last assertion is a consequence of the fact that t.d.l.c.~groups of rational discrete cohomological dimension zero are precisely the compact ones (cf.~Fact~\ref{fact:discoho}(3)). In particular, a Coxeter group has rational cohomological dimension zero precisely when it is spherical.
     \end{proof}
\begin{cor}\label{cor:ihara}
    Let $G$ be a unimodular t.d.l.c.~group acting on a locally finite building $\Delta$ of type~$(W,S)$ with compact stabilisers and finitely many orbits. If $W$ is infinite and virtually free, then $G$ is non-compact and has a finitely generated free
subgroup that is cocompact and discrete.
\end{cor}
 \begin{proof}
   By~\cite[Corollary~1.2]{dun:acc}, $\ccd_\Q(W)=1$. Hence, by Theorem~\ref{thm:ihbuil}, $\ccd_\Q(G)=1$ as well. Moreover, $G$ acts on~$\DeltaDav$ with compact open stabilisers and finitely many orbits and so~$G$ is compactly presented (cf.~\cite[Theorem~4.9]{ccc}).
   Thus $G$ is the fundamental group of a finite graph of profinite groups (cf.~\cite[Theorem~B]{cast:cd1}). As $G$ is unimodular, \cite[Theorem~16]{km:cayley} yields the claim.
 \end{proof}
By Proposition~\ref{fact:Gunimod}, the unimodularity assumption in Corollary~\ref{cor:ihara} can be dropped in case that the $G$-action on~$\Delta$ is chamber-transitive. 
Combining Lemma~\ref{lem:coho} and Corollary~\ref{cor:end3}, one deduces the following.
\begin{cor}\label{cor:eDelta}
Let $\Delta$ be a locally finite building and let $G$ be a t.d.l.c.~group acting on $\Delta$ with compact stabilisers and finitely many orbits. Then
\begin{equation}\label{eq:eG2}
        e(|\DeltaDav|)=e(G)=1-\dim_\Q \Hc^0(|\Delta_{\mathrm{Dav}}|, \Q)+\dim_\Q \Hc^1(|\Delta_{\mathrm{Dav}}|,\Q),
    \end{equation}
 where $ e(|\DeltaDav|)$ is defined as in~\cite[Appendix~G.3]{dav:book}. In particular, if $G$ is non-compact, then
 $e(G)=1+\dim_\Q \Hc^1(|\Delta_{\mathrm{Dav}}|,\Q)$.
\end{cor}
\begin{proof} By Lemma~\ref{lem:coho} (together with the isomorphism $\BiG\to C_c(G,\Q)$ in~\cite[Equation~(4.64)]{cw:qrat}) and Corollary~\ref{cor:end3}, one has 
$$e(G)=1-\dim_\Q \Hc^0(|\Delta_{\mathrm{Dav}}|, \Q)+\dim_\Q \Hc^1(|\Delta_{\mathrm{Dav}}|,\Q)$$
and, if $G$ is non-compact, $e(G)=1+\dim_\Q \Hc^1(|\Delta_{\mathrm{Dav}}|,\Q)$ (see~\eqref{eq:dH0}).
    It suffices now to observe that $G$ acts geometrically on the proper geodesic space $|\Delta_{\mathrm{Dav}}|$ and then $e(G)=e(|\Delta_{\mathrm{Dav}}|)$ (cf.~\cite[Theorem~4.C.5]{CH}).
\end{proof}

\begin{rem}\label{rem:eG0}
    Let $G$ be a t.d.l.c.~group acting on a locally finite building $\Delta$ of type~$(W,S)$ with compact stabilisers and finitely many orbits.
   Recall from Fact~\ref{fact:discoho}(3) that $\ccd_\Q(G)=0$ if and only if~$G$ is compact.
    By this fact and Theorem~\ref{thm:ihbuil}, one has
    \begin{eqnarray*}
        e(G)=0 &\Longleftrightarrow & G \text{ compact }\Longleftrightarrow 
        \ccd_\Q(G)=0 \\
        &\Longleftrightarrow &\ccd_\Q(W)=0 \Longleftrightarrow W\text{ finite }\Longleftrightarrow e(W)=0.
    \end{eqnarray*}
\end{rem}
The following proposition provides an alternative proof of~\cite[Proposition~5.20]{cmr:KMgrps}, and is based on the main theorem of~\cite{compsup}.
\begin{prop}\label{prop:ecd}
  Let $G$ be a non-compact t.d.l.c.~group acting on a locally finite building~$\Delta$ of type~$(W,S)$ with compact stabilisers and finitely many orbits.
  Then $e(W)\leq e(G)$ and \begin{equation}\label{eq:eG=1}e(W)=1\Longleftrightarrow e(G)=1.\end{equation}
\end{prop}
Before proceeding to the proof of the proposition, we prove a lemma.
\begin{lem}\label{lem:proper}
    Let $\Delta$ be a locally finite building of type~$(W,S)$ and let~$\Sigma$ be any apartment of~$\Delta$. Fix a chamber~$c$ in~$\Sigma$ and consider the retraction $\rho=\rho_{c,\Sigma}\colon \Delta\to \Sigma$  onto~$\Sigma$ with centre~$c$ (cf.~\cite[\S 18.3, p.~333]{dav:book}). Then~$\rho\circ\iota=\mathrm{id}_\Sigma$ where $\iota\colon\Sigma\hookrightarrow \Delta$ is the inclusion map. Moreover, the induced maps $\bar\iota\colon|\Sigma_{\mathrm{Dav}}|\to|\DeltaDav|$ and $\overline{\rho}\colon |\DeltaDav|\to |\Sigma_{\mathrm{Dav}}|$ are continuous and proper (with respect to the weak topology).
\end{lem}
\begin{proof} By construction, $\rho\circ\iota=\mathrm{id}_\Sigma$ (cf.~\cite[p.~333]{dav:book}). Since $\iota$ and $\rho$ are morphisms of chamber systems (cf.~\cite[Lemma~18.1.12]{dav:book}), they induce the  maps of spaces $\bar\iota$ and $\bar\rho$, respectively, which take chambers isometrically to chambers (cf.~\cite[p.~338]{dav:book}). Clearly, $\bar\iota$ is continuous and proper. \\
\indent Now recall that $\DeltaDav$ and~$\Sigma_{\mathrm{{Dav}}}$ are locally finite simplicial complexes (cf.~Section~\ref{sss:dav}). 
Hence, by~\cite[Theorem~3.11]{kramer}, $|\DeltaDav|$ and~$|\Sigma_{\mathrm{Dav}}|$ are proper metric spaces with respect to the intrinsic metrics~$d_\Delta$ and~$d_{\Sigma}$ induced by the simplicial structures on~$\DeltaDav$ and~$\Sigma_{\mathrm{Dav}}$ (cf.~\cite[Definition~3.9]{kramer}) and such a metric topology on them coincides with the weak topology. Thus,  every compact subset of $|\Sigma_{\mathrm{Dav}}|$ is contained in a closed ball centred in $x$, where $x$ denotes the point in~$|\Sigma_{\mathrm{Dav}}|$ corresponding to the chamber $c$.
By~\cite[Lemma~18.3.5]{dav:book}, for every $y\in |\DeltaDav|$ one has $d_\Sigma(x,\bar\rho(y))=d_\Delta(\bar\iota(x),y)$. Hence, the preimage $\bar\rho^{-1}(K)$ of an arbitrary compact subset $K$ of $|\Sigma_{\mathrm{Dav}}|$ is contained in a closed ball centred in $\bar\iota(x)$ and $\bar\rho$ is proper.
By~\cite[Lemma~18.3.3]{dav:book}, $\bar\rho$ is 1-Lipschitz and thus continuous. 
\end{proof}

\begin{proof}[{Proof of Proposition~\ref{prop:ecd}}]
Recall that $\mathrm H^\bullet_c(\_,\Q)$ (see for instance~\cite[end of p.~266]{geog} for the definition) is a contravariant functor from the category of oriented CW-complexes of locally finite type and CW-proper homotopy classes of CW-proper maps (cf.~\cite[p.~263 and {Theorem~12.2.1}]{geog}). By Lemma~\ref{lem:proper}, both $\bar\iota$ and $\bar\rho$ are CW-proper (cf.~\cite[p.~227]{geog}) and so, for every $k\geq 0$, 
$$\bar\rho^*_k:=\mathrm H^k_c(\bar\rho,\Q)\colon \mathrm H^k_c(|\Sigma_{\mathrm{Dav}}|,\Q)\to \mathrm H^k_c(|\DeltaDav|,\Q)$$
is a coretraction (i.e., a split monomorphism).
Hence, since~$\Sigma_{\mathrm{Dav}}$ is simplicially isomorphic to~$\Sigma(W,S)_{\mathrm{Dav}}$, for every~$k\geq 0$ one has
\begin{equation*}
\dim_\Q\Hc^k(|\Sigma(W,S)_{\mathrm{{Dav}}}|, \Q)=\dim_\Q\Hc^k(|\Sigma_{\mathrm{{Dav}}}|, \Q)\leq \dim_\Q\Hc^k(|\DeltaDav|,\Q).
\end{equation*}
Therefore, by Corollary~\ref{cor:eDelta}, 
  $$e(W)\leq 1+\dim_\Q\Hc^1(|\Sigma(W,S)_{\mathrm{{Dav}}}|, \Q)\leq 1+\dim_\Q\Hc^1(|\DeltaDav|,\Q)=e(G).$$
To prove the equivalence, one first notices that $e(G)=1+\dim_\Q\Hc^1(|\DeltaDav|, \Q)$ and $e(W)=1+\dim_\Q\Hc^1(|\Sigma(W,S)_{\mathrm{Dav}}|,\Q)$ (cf.~Corollary~\ref{cor:eDelta} and Remark~\ref{rem:eG0}). Hence, \eqref{eq:eG=1} holds once the following is proved: 
  \begin{equation}\label{eq:vanish}
     \Hc^1(|\Delta_{\mathrm{Dav}}|, \Q)=0 \Longleftrightarrow \Hc^1(|\Sigma(W,S)_{\mathrm{Dav}}|, \Q)=0.
     \end{equation}
 In fact~\eqref{eq:vanish} holds for every degree $k\geq 0$. Indeed, by the main theorem of~\cite{compsup} one has
     \begin{equation}\label{eq:c3}
     \Hc^k(|\Delta_{\mathrm{Dav}}|, \Q)\cong \bigoplus_{T\subseteq S, \atop T\text{ spherical}}\mathrm{H}^k(K, K^{S\setminus T})\otimes \widehat{A}^T(\Delta)
     \end{equation}
     and 
      \begin{equation}\label{eq:c4}
    \Hc^k(|\Sigma(W,S)_{\mathrm{Dav}}|, \Q)\cong \bigoplus_{T\subseteq S, \atop T\text{ spherical}}\mathrm{H}^k(K, K^{S\setminus T})\otimes \widehat{A}^T(\Sigma(W,S)),
     \end{equation}
     for every $k\geq 0$. Here~$K$ denotes the Davis chamber, which depends only on the Coxeter group~$(W,S)$ and not on the building. Moreover, $\widehat{A}^T(\Delta)$ and~$\widehat{A}^T(\Sigma(W,S))$ are suitable abelian groups defined in~\cite[p.~570 and Definition~7.4]{compsup}. The relevant property is that~$\widehat{A}^T(\Delta)$ and~$\widehat{A}^T(\Sigma(W,S))$ are non-vanishing free abelian groups for every spherical subset~$T\subseteq S$. 
     Indeed, they are defined as images of non-trivial free abelian groups under injective morphisms 
     (cf.~\cite[Remark at p.~570]{compsup}).
Therefore, $\Hc^k(|\Delta_{\mathrm{Dav}}|, \Q)=0$ if and only if $\mathrm{H}^k(K, K^{S-T})=0$ for every~$T\subseteq S$ spherical (cf.~\eqref{eq:c3}). By~\eqref{eq:c4}, an analogous equivalence holds for $\Hc^k(|\Sigma(W,S)_{\mathrm{Dav}}|, \Q)$. 
\end{proof}

\begin{rem}\label{rem:no=}
   In Proposition~\ref{prop:ecd} the equality between $e(W)$ and $e(G)$ does not hold in general. For example, consider $G=\SL_2(\Q_p)$ acting on its Bruhat--Tits building~$\Delta$. 
   It is well-known that the Davis complex of~$\Delta$ is the first barycentric subdivision of a $(p+1)$-regular tree and, by~Corollary~\ref{cor:eDelta}, $e(\SL_2(\Q_p))=\infty$.
   On the other hand, the Weyl group associated to $\Delta$ is the infinite dihedral group, which is two-ended.
\end{rem}


\subsection{A relation between the flat rank and $\ccd_\Q(G)$}\label{s:flatrk}
Recall that the {\em scale function}~$s\colon G\to \Z$ of a t.d.l.c.~group~$G$ is defined by 
$$s(g):=\min\{|g\caO g^{-1}:\caO\cap g\caO g^{-1}|\colon \caO\leq G\text{ compact open}\}.$$
A subgroup $H$ of $G$ is said to be  {\em flat} if there is a compact open subgroup $U\leq G$ such that $|hUh^{-1}\colon U\cap hUh^{-1}|=s(h)$ for all $h\in H.$ 
 For such a subgroup $H$ the set
$$H(1) := \{h \in H \mid s(h) = 1 = s(h^{-1})\}$$ is a normal subgroup of $H$ and $H/ H(1)$ is free abelian (cf.~\cite[Corollary~6.15]{wil04}). Then one defines the \emph{flat rank}~of~$G$ by
\begin{equation}\label{eq:flatrank}
    \flatrk(G):=\sup\{\mathrm{rk}_\Z(H/ H(1))\mid \text{$H\leq G$ flat subgroup}\}.
    \end{equation}
    The notion of flat rank was introduced in~\cite{wil04}, although it was not given that name there. For further details, the reader is referred to~\cite{brw:build,wil04}.

Following~\cite[\S 3.5]{cw:qrat}, a t.d.l.c.~group $G$ is said to be \textit{$N$-compact} if, for every compact open subgroup $\caO$ of $G$, the normaliser $N_G(\caO)$ is compact. By~\cite[Proposition~3.10(b)]{cw:qrat}, for such groups one has
    \begin{equation}\label{eq:ncompact}
        \flatrk(G)\leq\ccd_\Q(G).
    \end{equation}
The result below shows that, given a locally finite thick building $\Delta$, every t.d.l.c.~group acting  Weyl-transitively on $\Delta$ with compact open stabilisers is $N$-compact.
As a consequence of \eqref{eq:ncompact}, one  deduces that $\flatrk(G)\leq\ccd_\Q(G)$  for such t.d.l.c.~groups. 
\begin{prop}\label{prop:bcg}
Let $G$ be a t.d.l.c.~group acting Weyl-transitively on a locally finite thick building $\Delta$ with compact stabilisers. Then $G$ has the double coset property.  In particular, it is $N$-compact.
     \end{prop}
     Before proving the result, for the reader's convenience we recall some terminology and notation.
     Let~$\caO$ be a compact open subgroup of a t.d.l.c.~group~$G$. If~$G$ is unimodular, one observes that
     \begin{equation}
         |N_G(\caO):\caO|=|\{\caO g\caO\in \caO\backslash G/\caO\mid \mu_\caO(\caO g\caO)=1\}|,
     \end{equation}
    where $\mu_{\caO}$ denotes the Haar measure of $G$ such that $\mu_\caO(\caO)=1$. Indeed, one has $$|\caO:\caO\cap g\caO g^{-1}|=\mu_\caO(\caO g\caO)=\mu_\caO(\caO g^{-1}\caO)=|\caO:\caO\cap g^{-1}\caO g|.$$ In other words, for unimodular t.d.l.c.~groups, the index $|N_G(\caO):\caO|$ coincides with $R_\caO(1)$, where $R_\caO(n):=|\{\caO g\caO\in \caO\backslash G/\caO\mid \mu_\caO(\caO g\caO)=n\}|.$
     Following~\cite[\S 6]{ccw:zeta}, a t.d.l.c.~group $G$ is said to have the {\em double coset property} if $R_\caO(n)<\infty$ for every $n\geq 1$. According to~\cite[Proposition~6.2]{ccw:zeta}, this property does not depend on $\caO$.
     \begin{proof} By~\cite[Proposition~6.2]{ccw:zeta}, it suffices to show that $G$ has the double coset property for the stabiliser $\caO$ of a chamber~$c_0$ in $\Delta$.
        By Weyl-transitivity, there is a bijection between $\caO\backslash G/\caO$ and $W$ and then one may write $G=\bigsqcup_{w\in W}\caO w\caO$ (cf.~\cite[\S 6.1.4]{ab:build}). 
        {For $s\in S$, let $q_s=|\caO:\caO\cap s\caO s^{-1}|$.} 
        By the orbit-stabiliser theorem, $q_s$ is the order of the $\caO$-orbit of~$s\cdot c_0$ which, by Weyl-transitivity, is the set of all chambers in~$\Delta$ distinct to~$c_0$ that are $s$-adjacent to~$c_0$. Since~$\Delta$ is thick, $q_s\geq 2$ for every~$s\in S$.
        For every~$w\in W$, fix a set of representatives~$\mathrm{Rep}(\caO w\caO/\caO)\subseteq G$ for the left $\caO$-cosets contained in~$\caO w\caO$.  Hence, for every reduced word $w=s_1\dotsm s_n$ in $(W,S)$ we have
    \begin{equation*}\begin{split}\mu_\caO(\caO w\caO)&=\mu_\caO\Bigg(\bigsqcup_{r\in \mathrm{Rep}(\caO w\caO/\caO)}rw\caO\Bigg)=\sum_{r\in \mathrm{Rep}(\caO w\caO/\caO)}\mu_\caO(rw\caO)=\\
    &=|\mathrm{Rep}(\caO w\caO/\caO)|=|\caO:\caO\cap w\caO w^{-1}|=q_{s_1}\dotsm q_{s_n},\end{split}\end{equation*}
    where the second and third equalities follow from the fact that~$\mu_\caO$ is a left-invariant Haar measure on~$G$, and the last equality follows from~\cite[Proposition~4]{brw:build}.
    For every reduced word~$w=s_1\dotsm s_n$ in~$(W,S)$, let $q_w:=q_{s_1}\dotsm q_{s_n}$. Let also~$P_{(W,S)}(t):=\sum_{w\in W}t^{\ell(w)}$, where~$\ell(w)$ is the word length of~$w$ in~$(W,S)$, be the growth series of~$(W,S)$ (cf.~\cite[\S 17.1]{dav:book}). It is well known that~$P_{(W,S)}(t)$ converges at~$t=2^{-k}$ for all $k\gg 1$ (cf.~\cite[p.~322]{dav:book}). Hence, since $q_s\geq 2$ for every $s\in S$, there is $k\geq 1$ such that
    \begin{equation}\label{eq:convWS}
        \sum_{w\in W}q_w^{-k}\leq \sum_{w\in W}2^{-k\ell(w)}<\infty.
    \end{equation}
    On the other hand, since $q_w^{-k}>0$ for all $w\in W$, for every $n\geq 1$ one deduces that
      $$\infty>\sum_{w\in W}q_w^{-k}\geq \sum_{w\in W:\atop q_w=n}n^{-k}=|\{w\in W\mid q_w=n\}|\cdot n^{-k}.$$ 
      As a conclusion, the cardinality $\caR_\caO(n)=|\{w\in W\mid q_w=n\}|$ is finite for~$n\geq 1$. 

      Since $G$ is unimodular by Proposition~\ref{fact:Gunimod}, $G$ is $N$-compact.
     \end{proof}
     For Weyl-transitive closed subgroups of the t.d.l.c.~group~$\Aut(\Delta)$ (cf.~Section~\ref{sss:build}), one can provide an alternative proof of \eqref{eq:ncompact} based on the relation between the invariants of $G$ and those of $W$. Recall that the \textit{algebraic rank} of~$W$ is
defined to be the supremum of the ranks of the free abelian subgroups of $W$ and is
denoted $\algrk(W)$. It is worth recalling that, in~\cite[Theorem~6.8.3]{krammer}, D.~Krammer proved that its computation can be reduced to the study of combinatorial properties of the associated Coxeter diagram. 
For future reference, we record the following well-known fact:
\begin{lem}\label{lem:algvcd}
    For every Coxeter group~$(W,S)$ one has
    \begin{equation*}
    \algrk(W)\leq\ccd_\Q(W)\leq\vcd(W)\leq |S|-1.
\end{equation*}
\end{lem}
\begin{proof}
     Since the rank of a free abelian group coincides with its rational cohomological dimension and $\ccd_\Q(H)\leq \ccd_\Q(G)$ whenever $H\leq G$, one has~$\algrk(W)\leq\ccd_{\Q}(W)$. For the remaining inequalities, see~\eqref{eq:cd_ineq}.
\end{proof}
     \begin{thm}\label{thm:fl-cd}
        Let~$\Delta$ be a locally finite building and~$G$ a Weyl-transitive closed subgroup of~$\Aut(\Delta)$. Then
        $$\flatrk(G)\leq\algrk(W)\leq\ccd_{\Q}(W)=\ccd_\Q(G).$$
     \end{thm}
    \begin{proof}
        By~\cite[Theorem~A]{brw:build}, $\flatrk(G)\leq \mathrm{rk}(|W|_0)$ where~$\mathrm{rk}(|W|_0)$ is the maximal dimension of flats in the Davis realisation of the Coxeter complex~$\Sigma(W,S)$. In turn,~\cite[Corollary~C]{caphag} states that~$\mathrm{rk}(|W|_0)=\algrk(W)$.
       Therefore, $$\flatrk(G)\leq\algrk(W).$$
   Now Theorem~\ref{thm:ihbuil} and {Lemma~\ref{lem:algvcd}} yield the claim.
     \end{proof}
    Let $\mathbb K$ be a non-Archimedean local field and let~$\boG$ be a semisimple, simply connected algebraic group defined over $\mathbb K$ with affine Weyl group~$W$. If $\boG(\mathbb K)$ denotes the group
of $\mathbb K$-rational points {of $\boG$}, then 
\begin{equation}\label{eq:alggrps}
    \algrk(W)=\flatrk(\boG(\mathbb K)) = \ccd_\Q(\boG(\mathbb K))=\ccd_\Q(W),
\end{equation}
cf.~\cite[Corollary~18]{brw:build} and~\cite[Example~3.11]{cw:qrat}. 
\begin{rem}
For $p$ prime, $(\mathrm{PSL}_{d+1}(\Q_p))_{d\geq 1}$ provides a sequence of topologically
simple compactly generated t.d.l.c.~groups having arbitrary rational discrete cohomological dimension $d$. 
\end{rem}

Another relevant source of examples arises from groups~$G$ with a \emph{root group datum~$(U_\alpha)_{\alpha\in \Phi}$ of type~$(W,S)$ and finite root groups~$(U_\alpha)_{\alpha\in \Phi}$}. Namely,~$(U_\alpha)_{\alpha\in \Phi}$ is a family of finite subgroups of~$G$ called \emph{root groups} and indexed by the set of all roots~$\Phi$ of~$(W,S)$ and subject to a series of axioms (see~\cite[\S 2.1]{CapRem:root}). Prominent examples of groups admitting a root group datum with finite root groups are Kac--Moody groups~$G(\F)$ over a finite field~$\F$ (see~\cite[\S 8.11.5]{ab:build}).
For the purposes of the paper, it suffices to recall only the following: for every root group datum~$(G, (U_\alpha)_{\alpha\in \Phi})$ of type~$(W,S)$ and finite root groups, there is a standard way to produce two BN-pairs~$(B_+,N)$ and~$(B_-,N)$ of~$G$ of type~$(W,S)$ (see~\cite[\S 4.1.2]{CapRem:root}). For this paper, it suffices to focus on one of them, say~$(B_+,N)$. 
Let~$\Delta_+$ be the locally finite building associated to it (see~\cite[Definition~6.32]{ab:build}). The finiteness assumption on each~$U_\alpha$ guarantees that~$\Delta_+$ is locally finite; see~\cite[Proposition~8.56(2)]{ab:build}. Denote by $\bar G$ the closure of the image of $G$ in $\Aut(\Delta_+)$.
The topological group $\bar G$ is called {\em geometric completion of $G$} (cf.~\cite[p.~187]{marq:km}). 
For example, if $G$ is an abstract Kac--Moody group defined over a finite field~$\mathbb{F}$, then $\overline{G(\mathbb{F})}$ gives the so-called {\em complete Kac--Moody group} \cite[\S 1.B.2]{rr06}. For every geometric completion~$\bar{G}$ one has $\algrk(W)=\flatrk(\bar G)$ as well (cf.~\cite[Theorems~A~and~B]{brw:build} and~\cite[Corollary~C]{caphag}). 
     In particular, from Theorem~\ref{thm:ihbuil} and Theorem~\ref{thm:fl-cd} one deduces the following.
\begin{prop}\label{prop:algcd}
Let $G$ be a group with a root group datum of type~$(W,S)$ and finite root groups, and denote by $\bar G$ the geometric completion of $G$. Then $\algrk(W)=\ccd_\Q(W)$ if and only if $\flatrk(\bar G)=\ccd_{\Q}(\bar G)$.
\end{prop}
This raises the following question.
 \begin{ques}\label{q:algrkcdQ}
  For which Coxeter groups does $\algrk(W)=\ccd_\Q(W)$ hold?
\end{ques}
For instance, every affine Coxeter group $(W,S)$ has algebraic rank $|S|-1$ and so the equality $\algrk(W)=\ccd_\Q(W)$ holds (see the end of Section~\ref{sus:vcd}).
Any product  of finitely many Coxeter groups, say $W_1\times W_2\times\dotsm \times W_n=:W$, satisfies~$\algrk(W)=\ccd_\Q(W)$ provided that~$\algrk(W_i)=\ccd_{\Q}(W_i)$ for every $i=1,\dots,n$ (cf.~\cite[Theorem~5.6]{bieri}). Another example is given by M.~W.~Davis and T.~Januszkiewicz in the main theorem of~\cite{dj:rightangled}: for each right-angled Artin group $A(\Gamma)$ there is a right-angled Coxeter group $W_{\Gamma}$ which contains $A(\Gamma)$ as a subgroup of finite index. In particular, $\ccd_\Q(W_{\Gamma})=\ccd_\Q(A(\Gamma))$ and $\algrk(W_{\Gamma})=\algrk(A(\Gamma))$. {It suffices now to notice that $\algrk(A(\Gamma))$ coincides with the largest size of a clique of $\Gamma$, which is also the dimension of the associated Salvetti complex (cf.~\cite[Theorems~3.4 and~3.5]{koberda}). In particular, $\algrk(W_\Gamma)=\ccd_\Q(W_\Gamma)=\vcd(W_\Gamma)$. In Appendices~\ref{s:hyper} and~\ref{app:maxrk}, we answer Question~\ref{q:algrkcdQ} for Coxeter groups of hyperbolic type (see Propositions~\ref{prop:vcdhyp} and~\ref{prop:noncomp}) and for Coxeter groups of rank~$3$ (see Proposition~\ref{prop:3gen}).

\smallskip

In~\cite[Theorem~26]{brw:build}, for every natural number $n\geq 1$, the authors constructed a nonlinear, topologically simple, compactly generated t.d.l.c.~group $\bar G_n$ of flat rank~$n$. For these groups we can compute the rational discrete cohomological dimension and obtain the following:
\begin{prop}
The family $\{\bar G_n\}_{n\geq 1}$ of non-linear, topologically
simple, compactly generated, t.d.l.c.~groups satisfies
 $\ccd_\Q(\bar G_1)=2$ and $\ccd_\Q(\bar G_n)=n$  for all $n\geq 2$.
\end{prop}
\begin{proof}
    For every integer $n \geq 1$, let~$\bar{G}_n$ be the non-linear, topologically
simple, compactly generated, t.d.l.c.~group of
flat rank $n$ as constructed in~\cite[Theorem~26]{brw:build}. Namely, $\bar{G}_n$ is the topological Kac--Moody group over a given finite field associated to a Kac--Moody root datum whose type is the Coxeter diagram~$Y_n$ defined inductively as follows:
\begin{itemize}
    \item[-] let~$Y_1$ be a cycle of length~$5$ all of whose edges are labelled with~$\infty$,
    \item[-] if~$n > 1$ let~$Y_n$ be the diagram obtained by joining every vertex of~$Y_1$ to
every vertex of a diagram of type $\tilde A_n$ by an edge labelled~$\infty$.
\end{itemize}
Denote by $(W_n,\Sigma_n)$ the Coxeter group associated to the Coxeter diagram $Y_n$.
 By Theorem~\ref{thm:ihbuil}, for every $n\geq 1$ we compute~$\ccd_\Q(\bar G_n)$ by computing~$\ccd_\Q(W_n)$. 
Since the maximal spherical subsets of~$\Sigma_1$ have size~$2$, the Bestvina complex $B_\Z(W_1,\Sigma_1)$ is $2$-dimensional and so $\ccd_\Q(W_1)\leq \vcd(W_1)=2$ (cf.~Eq.~\eqref{eq:cd_ineq}). Since free groups have integral cohomological dimension at most one, $W_1$ is not virtually free. However, by~\cite[Corollary~1.2]{dun:acc}, a finitely generated group~$G$ satisfying~$\ccd_\Q(G)\leq 1$ must be virtually free. Therefore, $\ccd_\Q(W_1)=2$.
For $n\geq 2$, let~$W(\tilde A_n)$ be the Coxeter group of type~$\tilde A_n$. Recall from the end of Section~\ref{sus:vcd} that~$\ccd_\Q(W(\tilde A_n))=n$.
Since $W_n\cong W_1\ast W(\tilde A_n)$, by~\cite[\S II.6, Proposition~6.1]{bieri} we have~$\ccd_\Q(W_n)=\max\{\ccd_\Q(W_1),\ccd_\Q(W(\tilde A_n))\}=n$. Therefore, 
 $\ccd_\Q(\bar G_1)=2$ and $\ccd_\Q(\bar G_n)=n$ for all $n\geq 2$.
\end{proof}
\section{Rational discrete cohomological dimension one}
In this section, we focus on t.d.l.c.~groups~$G$ that act chamber-transitively on a locally finite building $\Delta$. We will show that $\ccd_\Q(G)\leq 1$ implies that $G$ acts with compact open vertex stabilisers on a tree which can be explicitly constructed from the building~$\Delta$ (cf.~Theorem~\ref{thm:cd1}). The latter tree is an instance of a more general construction, introduced in Section~\ref{s:hp} and that generalises~\cite[\S 4.2]{hp:arb}.
\subsection{Haglund--Paulin trees}\label{s:hp}
Let $\Delta=(\caC, \delta)$ be a building of type $(W,S)$. For $T\subseteq S$ and $\caM\subseteq\caC$, set
\begin{equation}\label{eq:sigmaM}
	\caR_T(\caM):=\{\Res_T(c) \mid c\in \caM\},
\end{equation}
(cf.~Section~\ref{sss:build}). If necessary, one writes $\Res_T^\Delta(c)$ and $\caR_T^\Delta(\caM)$ in place of $\Res_T(c)$ and $\caR_T(\caM)$, respectively.
\begin{defi}\label{def:indgraph}
 Let $\Delta=(\caC, \delta)$ be a building of type $(W,S)$ and $(\caG, \Lambda)$ be a graph of special subgroups of $(W,S)$ (cf.~Definition~\ref{def:visualdec}). For every non-empty subset $\caM\subseteq \caC$  define a Serre-graph $\Gamma_{(\caG, \Lambda)}^\Delta(\caM)=(\euV(\caM), \euE(\caM))$  as follows:
   \begin{equation*}
       \euV(\caM):=\bigsqcup_{v\in \euV(\Lambda)}\caR_{S_v}(\caM) \quad \text{and}\quad \euE(\caM):=\bigsqcup_{e\in \euE(\Lambda)}\caR_{S_e}(\caM).
   \end{equation*}
   For every $e\in \euE(\Lambda)$, although $S_e=S_{\bar{e}}$, here $\caR_{S_e}(\caM)$ and $\caR_{S_{\bar{e}}}(\caM)$ are regarded as disjoint copies of the same set.
   The edge-reversing, the origin and the terminus maps are defined, respectively, as 
	\begin{equation}\label{eq:ot}
 \overline{\Res_{S_e}(c)}:=\Res_{S_{\bar{e}}}(c),\quad
		o(\Res_{S_e}(c)):=\Res_{S_{o(e)}}(c),\quad 
            t(\Res_{S_e}(c)):=\Res_{S_{t(e)}}(c),
	\end{equation}
   for all $e\in \euE(\Lambda)$ and $c\in \caM$.
\end{defi}
\begin{rem}
\begin{itemize}
    \item[(i)] The assignments in~\eqref{eq:ot} are well-defined. For the inversion map it is obvious. Moreover, for $e\in \euE(\Lambda)$ and $v\in \{o(e),t(e)\}$, observe that  $\Res_{S_e}(c)=\Res_{S_e}(d)$ implies $\delta(c,d)\in W_{S_e}\subseteq W_{S_v}$ and then, by~\eqref{eq:res=}, $\Res_{S_v}(c)=\Res_{S_v}(d)$.
    \item[(ii)] If $\Lambda$ is a combinatorial graph, so is $\Gamma=\Gamma_{(\caG, \Lambda)}^\Delta(\caM)$. In other terms, if the map $e\in \euE(\Lambda)\longmapsto (o(e), t(e))\in \euV(\Lambda)\times \euV(\Lambda)$ is injective, so is the map $\Res_{S_e}(c)\in \euE(\Gamma)\longmapsto (\Res_{S_{o(e)}}(c), \Res_{S_{t(e)}}(c))\in \euV(\Gamma)\times \euV(\Gamma)$. Indeed, $\Res_{S_{o(e)}}(c)\cap \Res_{S_{t(e)}}(c)\supseteq\Res_{S_e}(c)$ for all $e\in\euE(\Lambda)$ and $c\in \caM$. 
\end{itemize}
\end{rem}
\begin{example}[The 1-segment]\label{ex:indtree}
For a given Coxeter group $(W,S)$, let $S_\btd, S_\btu\subseteq S$ be subsets such that $S=S_\btd\cup S_\btu$. Put $S_\bullet=S_\btd\cap S_\btu$. Let $\Lambda$ be the $1$-segment graph (in the sense of J-P.~Serre) with edge-pair $\{e,\bar{e}\}$ and define a graph of groups $(\caG, \Lambda)$ by putting $\caG_{o(e)}=W_{S_\btd}$, $\caG_{t(e)}=W_{S_{\btu}}$ and $\caG_e=\caG_{\bar{e}}=W_{S_\bullet}$. Let $\Sigma=\Sigma(W,S)=(W, \delta_W)$ be the standard Coxeter complex of type $(W,S)$ (cf.~Section~\ref{sss:build}). Recall that, for every $\times\in \{\btd, \btu, \bullet\}$, the residue $\Res_{S_\times}^{\Sigma}(w)$ can be identified with the left-coset $wW_{S_\times}$.
\begin{itemize}
    \item[(i)] The graph $\Gamma_{(\caG, \Lambda)}^{\Sigma}(W)$ is isomorphic to the graph $X=X(W; W_{S_\btd}, W_{S_\btu})$ defined in~\cite[\S I.4, proof of Theorem~7]{ser:trees}. By~\cite[\S I.4, Theorems~6~and~7]{ser:trees}, $\Gamma_{(\caG, \Lambda)}^{\Sigma}(W)$ is a tree if and only if $W\cong W_{S_\btd}\ast_{W_{S_\bullet}}W_{S_\btu}$.
    \item[(ii)] More generally, let $\Delta=(\caC,\delta)$ be a building of type $(W,S)$. Then $\Gamma_{(\caG,\Lambda)}^\Delta(\caC)$ coincides with the graph introduced by F.~Haglund and F.~Paulin in~\cite[\S 4.2]{hp:arb}. By~\cite[p.~146, Remarques]{hp:arb} and~\cite[Lemme~4.3]{hp:arb} (or Corollary~\ref{cor:hp} below), $\Gamma_{(\caG,\Lambda)}^\Delta(\caC)$ is a tree if and only if $W\cong W_{S_\btd}\ast_{W_{S_\bullet}}W_{S_\btu}$.
From now on, the graph $\Gamma_{(\caG,\Lambda)}^\Delta(\caC)$ constructed here will be denoted by $\Gamma_{\trt}^\Delta(\caC)$.
\end{itemize}
\end{example} 
\begin{lem}\label{lem:indgraph}
    Let $\Delta=(\caC, \delta)$ be a building of type $(W,S)$, let $\caM\subseteq \caC$ be non-empty and $(\caG, \Lambda)$ be a graph of special subgroups of $(W,S)$. Denote by $\Sigma(W,S)$ the standard Coxeter complex of type~$(W,S)$ (see Section~\ref{sss:build}).
    \begin{itemize}
        \item[(i)] Every $W$-isometry $\alpha\colon W\longrightarrow \caC$ induces a graph isomorphism
        \begin{equation*}
            \varphi_\alpha: \Gamma_{(\caG, \Lambda)}^{\Sigma(W,S)}(W) \longrightarrow \Gamma_{(\caG,\Lambda)}^\Delta(\alpha(W))
        \end{equation*}
        mapping $wW_{S_x}$ to $\Res_{S_x}^{\Delta}(\alpha(w))$, for all $x\in \euV(\Lambda)\sqcup \euE(\Lambda)$ and $w\in W$.
        \item[(ii)] For $c\in \caC$, let $\caA_c$ be the set of all apartments in an atlas of $\Delta$ (cf.~Fact~\ref{fact:Wbuild}) containing $c$ as a chamber. Then,
        \begin{equation*}
\Gamma_{(\caG,\Lambda)}^\Delta(\caC)=\bigcup_{\Sigma\in\caA_c}\Gamma_{(\caG,\Lambda)}^\Delta(\caC(\Sigma)).
        \end{equation*}
    \end{itemize}
\end{lem}

\begin{proof}
Part~(ii) is straightforward from $\caC=\bigcup_{\Sigma\in \caA_c}\caC(\Sigma)$ (cf.~Fact~\ref{fact:Wbuild}(ii)). Thus it suffices to prove~(i).
For every $x\in \euV(\Lambda)\sqcup \euE(\Lambda)$, the assignment
        \begin{equation*}\label{eq:phia}
            \varphi_{\alpha, x}: wW_{S_x} \longmapsto \Res_{S_x}^\Delta(\alpha(w))\quad \forall\,w\in W
        \end{equation*}
        yields a well-defined bijection from $\caR_{S_x}^{\Sigma(W,S)}(W)$ to $\caR_{S_x}^\Delta(\alpha(W))$. Indeed, $\varphi_{\alpha, x}$ is well-defined and injective by the following: 
        \begin{equation*}
            \begin{array}{ccl}
                w_1W_{S_x}=w_2W_{S_x} & \Longleftrightarrow & \delta_W(w_1, w_2)=\delta(\alpha(w_1), \alpha(w_2))\in W_{S_x}\\
                                        & \Longleftrightarrow & \Res_{S_x}^\Delta(\alpha(w_1))=\Res_{S_x}^\Delta(\alpha(w_2)),
            \end{array}
        \end{equation*}
         for all $w_1,w_2\in W$ (cf.~\eqref{eq:res=}). 
        Finally, the surjectivity of $\varphi_{\alpha,x}$ is part of the definition.
\end{proof}
\begin{lem}\label{lem:galpath}
    Let $\Delta=(\caC, \delta)$ be a building of type $(W,S)$, $\caM\subseteq \caC$ be non-empty, and $(\caG, \Lambda)$ be a graph of special subgroups of $(W,S)$.
     Assume that $S=\bigcup_{v\in \euV(\Lambda)}S_v$.  Let $\gamma=(c_0, c_1, \dots, c_n)$ be a gallery in $\Delta$ with $c_i\in \caM$ for every $0\leq i\leq n$. Then, for all $v,w\in \euV(\Lambda)$, $\gamma$
     induces a path $\frp_{\gamma, v\rightarrow w}$ in $\Gamma_{(\caG, \Lambda)}^\Delta(\caM)$ from $\Res_{S_v}(c_0)$ to $\Res_{S_w}(c_n)$.
\end{lem}

\begin{proof}
   It suffices to show the statement for $\euE(\Lambda)\neq\emptyset$ and for a gallery~$\gamma=(c,d)$ with two chambers. That is, given $c,d\in \caM$ with $\delta(c,d)=s\in S$ and $v,w\in \euV(\Lambda)$, there is a path in $\Gamma_{(\caG,\Lambda)}^\Delta(\caM)$ from $\Res_{S_v}(c)$ to $\Res_{S_w}(d)$. By hypothesis, there is $z\in \euV(\Lambda)$ such that $s\in S_z$, and $\Res_{S_z}(c)=\Res_{S_z}(d)$ (cf.~\eqref{eq:res=}).
   If $v=w=z$, the path $\frp_{\gamma,v\to w}$ is the trivial path.
   In all the remaining cases, let $(e_1, \dots, e_n)$ be a path in $\Lambda$ from $v$ to $w$ running through $z$, i.e., $o(e_m)=z$ for some $1\leq m\leq n$.
   If $m=1$, then $v=z$ and $\Res_{S_v}(c)=\Res_{S_z}(c)=\Res_{S_z}(d)$ is the origin of $\Res_{S_{e_1}}(d)$ in $\Gamma_{(\caG, \Lambda)}^\Delta(\caM)$ (cf.~\eqref{eq:ot}).
   Thus the sequence 
   \begin{equation*}
       (\Res_{S_{e_1}}(d), \dots,\Res_{S_{e_n}}(d))
   \end{equation*}
   gives a path in $\Gamma_{(\caG, \Lambda)}^\Delta(\caM)$ from $\Res_{S_v}(c)$ to $\Res_{S_w}(d)$.
   Similarly, if $m\geq 2$ one has $t(e_{m-1})=z=o(e_{m})$ and then
   \begin{equation*}\label{eq:conn}
       \Res_{S_{t(e_{m-1})}}(c)=\Res_{S_z}(c)=\Res_{S_z}(d)=\Res_{S_{o(e_{m})}}(d).
   \end{equation*}
   Hence, the sequence
   \begin{equation*}\label{eq:indpath}
       (\Res_{S_{e_1}}(c), \dots,\Res_{S_{e_{m-1}}}(c),\Res_{S_{e_{m}}}(d), \dots, \Res_{S_{e_n}}(d))
   \end{equation*}
   defines a path in $\Gamma_{(\caG, \Lambda)}^\Delta(\caM)$ from $\Res_{S_v}(c)$ to $\Res_{S_w}(d)$.
\end{proof}

Let $\Delta=(\caC, \delta)$ be a building of type $(W,S)$.
A non-empty subset $\caM\subseteq \caC$ is \emph{gallery connected} if, for all $c,d\in \caM$, there exists a gallery $(c_0=c, c_1, \dots, c_k=d)$ in $\Delta$ with $c_i\in \caM$ for all $0\leq i\leq k$ (cf.~\cite[Definition~5.43]{ab:build}). For example, the set $\caC$ is gallery connected (cf.~\cite[Example~5.44(a)]{ab:build}). 

Lemma~\ref{lem:galpath} implies the following.
\begin{cor}\label{cor:indconn}
    Let $\Delta=(\caC, \delta)$ be a building of type $(W,S)$ and $(\caG, \Lambda)$ be a graph of special subgroups of $(W,S)$.
    Assume that $S=\bigcup_{v\in \euV(\Lambda)}S_v$. 
    Then, for every gallery connected subset $\caM\subseteq \caC$, the graph $\Gamma_{(\caG, \Lambda)}^\Delta(\caM)$ is connected. In particular, $\Gamma_{(\caG, \Lambda)}^\Delta(\caC)$ is connected. 
    
    Moreover, for every collection~$\{\Sigma_i=(\caC(\Sigma_i),\delta\vert_{\Sigma_i})\}_{i\in I}$ of apartments of $\Delta$,
    \begin{equation*}\label{eq:int}
            \Gamma_{(\caG,\Lambda)}^\Delta\Bigg (\bigcap_{i\in I}\caC(\Sigma_i)\Bigg )=\bigcap_{i\in I}\Gamma_{(\caG, \Lambda)}^\Delta(\caC(\Sigma_i)),
        \end{equation*}
\begin{equation*}\label{eq:uni}
            \Gamma_{(\caG,\Lambda)}^\Delta\Bigg (\bigcup_{i\in I}\caC(\Sigma_i)\Bigg )=\bigcup_{i\in I}\Gamma_{(\caG, \Lambda)}^\Delta(\caC(\Sigma_i)),
        \end{equation*}
       and if~$\bigcap_{i\in I}\caC(\Sigma_i)\neq \emptyset$ then $\Gamma_{(\caG,\Lambda)}^\Delta\Big(\bigcap_{i\in I}\caC(\Sigma_i)\Big)$ and $\Gamma_{(\caG,\Lambda)}^\Delta\Big(\bigcup_{i\in I}\caC(\Sigma_i)\Big)$ are both connected.      
\end{cor}

\begin{proof}
    For the first part of the statement, apply Lemma~\ref{lem:galpath}. The second part of the statement follows from the first part and the fact that, since $\bigcap_{i\in I}\caC(\Sigma_i)\neq\emptyset$, the sets $\bigcap_{i\in I}\caC(\Sigma_i)$ and $\bigcup_{i\in I}\caC(\Sigma_i)$ are gallery connected (cf.~\cite[Example~5.44(c)]{ab:build}).
\end{proof}

\begin{prop}\label{prop:indtree}
    Let $\Delta=(\caC, \delta)$ be a building of type $(W,S)$ and $(\caG, \Lambda)$ be a graph of special subgroups of $(W,S)$. Assume that $S=\bigcup_{v\in \euV(\Lambda)}S_v$. 
    Then $\Gamma_{(\caG, \Lambda)}^\Delta(\caC)$ is a tree if and only if $\Gamma_{(\caG, \Lambda)}^{\Sigma(W,S)}(W)$ is a tree. 
\end{prop}

\begin{proof}
    Let first $\Gamma_{(\caG, \Lambda)}^\Delta(\caC)$ be a tree and $\alpha: W\longrightarrow \caC$ be a $W$-isometry. Hence,~$\alpha(W)$ defines an apartment of $\Delta$. By Corollary~\ref{cor:indconn} one has that~$\Gamma_{(\caG, \Lambda)}^\Delta(\alpha(W))$ is connected and thus a tree. By Lemma~\ref{lem:indgraph}(i), also $\Gamma_{(\caG,\Lambda)}^{\Sigma(W,S)}(W)$ is a tree.

    Suppose conversely that $\Gamma_{(\caG,\Lambda)}^{\Sigma(W,S)}(W)$ is a tree. For $c\in \caC$, let $\caA_c$ be the collection of all the apartments in some atlas of $\Delta$ having $c$ as a chamber. By Lemma~\ref{lem:indgraph}, for every $\Sigma\in \caA_c$ the graph $\Gamma_\Sigma:=\Gamma_{(\caG,\Lambda)}^\Delta(\caC(\Sigma))$ is a subtree of $\Gamma:=\Gamma_{(\caG,\Lambda)}^\Delta(\caC)$, and $\Gamma=\bigcup_{\Sigma\in \caA_c}\Gamma_\Sigma$. In order to conclude that $\Gamma$ is a tree, it is sufficient to show that, for every finite non-empty subset $\caS\subseteq \caA_c$, the graph $\Gamma_\caS:=\bigcup_{\Sigma\in \caS}\Gamma_\Sigma$ is a subtree of $\Gamma$. Indeed, every closed path in $\Gamma$ without backtrackings must be contained in the union of finitely many $\Gamma_\Sigma$'s. 
    One proceeds by induction on the cardinality of $\caS$. For $|\caS|=1$ there is nothing to prove. Let now $\caS\subseteq \caA_c$ with $|\caS|\geq 2$ and suppose that the claim holds for all subsets of $\caA_c$ of cardinality $|\caS|-1$. Let $\Sigma\in \caS$ and $\caS':=\caS\setminus \{\Sigma\}$. By induction, both $\Gamma_\Sigma$ and $\Gamma_{\caS'}$ are trees. Moreover, by Corollary~\ref{cor:indconn}, 
    $$\Gamma_\Sigma\cap \Gamma_{\caS'}=\bigcup_{\Sigma'\in \caS'}(\Gamma_\Sigma\cap \Gamma_{\Sigma'})$$
    is a non-empty connected subgraph of $\Gamma_\Sigma$ and thus is a tree.
    By van Kampen theorem for complexes (cf.~\cite[Theorem~27.1]{bogo:gt}), for every point~$x$ in~$|\Gamma_\Sigma\cap \Gamma_{\caS'}|$, the fundamental group $\pi_1(|\Gamma_\caS|,x)$ is the free product of the groups $\pi_1(|\Gamma_\Sigma|,x)$ and $\pi_1(|\Gamma_{\caS'}|,x)$ amalgamated over the subgroup $\pi_1(|\Gamma_\Sigma\cap \Gamma_{\caS'}|,x)$. Thus, 
    $\Gamma_{\caS}=\Gamma_\Sigma\cup\Gamma_{\caS'}$ is a tree.
\end{proof}
The following theorem generalises results of F.~Haglund and F.~Paulin (cf.~Corollary~\ref{cor:hp}) to arbitrary visual graph of groups decompositions. 
\begin{thm}\label{thm:hpgen}
    Let $\Delta=(\caC, \delta)$ be a building of type $(W,S)$ and  $(\caG, \Lambda)$ be a graph of special subgroups of $(W,S)$. Then the following conditions are equivalent:
    \begin{itemize}
        \item[(i)] $(\caG, \Lambda)$ is a visual graph of groups decomposition of $(W,S)$;
         \item[(ii)] $S=\bigcup_{v\in \euV(\Lambda)}S_v$ and $\Gamma_{(\caG,\Lambda)}^\Delta(\caC)$ is a tree.
    \end{itemize}
\end{thm}
\begin{proof} 
By~\cite[Corollary~1.4.8(iii)]{bb:cox}, $S$ is a minimal generating set for~$W$.
Hence, if~$(\caG, \Lambda)$ is a visual graph of groups decomposition of $(W,S)$, then $W$ is generated by the union of the vertex groups and therefore $S=\bigcup_{v\in \euV(\Lambda)}S_v$. The statement is a consequence of Proposition~\ref{prop:indtree} and the main theorem of Bass--Serre theory (cf.~\cite[\S I.5.4, Theorem~13]{ser:trees}), because~(i) is equivalent to the fact that $\Gamma_{(\caG,\Lambda)}^{\Sigma(W,S)}(W)$ is a tree. Indeed, as $\caR_T^{\Sigma(W,S)}(W)=W/W_T$ for $T\subseteq S$, the graph $\Gamma_{(\caG,\Lambda)}^{\Sigma(W,S)}(W)$ is the universal Bass--Serre graph of $(\caG,\Lambda)$ (cf.~\cite[p.~51]{ser:trees}). 
\end{proof}
\begin{cor}[\protect{\cite[Lemme~4.3, Remarques at p.~146]{hp:arb}}]\label{cor:hp}
    Let $\Delta=(\caC,\delta)$ be a building of type $(W,S)$ and $S_\btd,S_\btu\subseteq S$ such that $S=S_\btd\cup S_\btu$. Put $S_\bullet=S_\btd\cap S_\btu$ and let $\Gamma_{\trt}^\Delta(\caC)$ be as in Example~\ref{ex:indtree}(ii). Then the following conditions are equivalent:
    \begin{itemize}
        \item[(i)] $S=S_\btd\cup S_\btu$ is an $\infty$-decomposition of $(W,S)$;
        \item[(ii)] $W\cong W_{S_\btd}\ast_{W_{S_\bullet}}W_{S_\btu}$;
        \item[(iii)] $\Gamma_{\trt}^\Delta(\caC)$ is a tree.
    \end{itemize}
\end{cor}

\begin{proof}
   The statement follows by combining Theorem~\ref{thm:hpgen} and the fact that $S=S_\btd\cup S_\btu$ is an $\infty$-decomposition of $(W,S)$ if and only if $W\cong W_{S_\btd}\ast_{W_{S_\bullet}}W_{S_\btu}$ (cf.~Section~\ref{ss:infdec}).
\end{proof}
Let $\Delta=(\caC,\delta)$ be a building of type $(W,S)$ and $(\caG,\Lambda)$ be a graph of special subgroups of $(W,S)$. Then every group $G$ acting on $\Delta$ has an induced action on $\Gamma_{(\caG,\Lambda)}^\Delta(\caM)$, for every (non-empty) $G$-invariant subset $\caM\subseteq \caC$. 

In preparation for the next corollary, we recall some basic definitions from~\cite[\S 5.1]{cmw:stasw}. A \emph{graph of t.d.l.c.~groups}~$(\widetilde{\caG},\Lambda)$ over a connected graph~$\Lambda$ is a graph of groups in which~$\widetilde{\caG}_\lambda$ is a t.d.l.c.~group for every~$\lambda\in \euV(\Lambda)\sqcup \euE(\Lambda)$, and each monomorphism~$\widetilde{\caG}_e\hookrightarrow \widetilde{\caG}_{t(e)}$, $e\in \euE(\Lambda)$, is open and continuous. Moreover, the fundamental group~$\pi_1(\widetilde{\caG},\Lambda)$ has a group topology that makes all the canonical inclusion~$\widetilde{\caG}_\lambda\hookrightarrow \pi_1(\widetilde{\caG},\Lambda)$, $\lambda\in \euV(\Lambda)\sqcup \euE(\Lambda)$, open and continuous. From now on, we will always consider $\pi_1(\widetilde{\caG},\Lambda)$ endowed with this topology. In particular, $\pi_1(\widetilde{\caG},\Lambda)$ is a t.d.l.c.~group.

\begin{cor}\label{cor:G decomp}
    Let $G$ be a $\sigma$-compact t.d.l.c.~group that acts chamber-transitively on a building $\Delta=(\caC,\delta)$ of type $(W,S)$, and let $c\in \caC$. Suppose that $(W,S)$ admits a visual graph of groups decomposition $(\caG,\Lambda)$. Then $G$ is topologically isomorphic to the fundamental group of the tree of t.d.l.c.~groups $(\widetilde{\caG},\Lambda)$, where $\widetilde{\caG}_\lambda$ is the setwise stabiliser of $\Res_{S_\lambda}(c)$ for every $\lambda\in \euV(\Lambda)\sqcup \euE(\Lambda)$, and the continuous open monomorphisms $\widetilde\caG_e\hookrightarrow \widetilde\caG_{t(e)}$, $e\in\euE(\Lambda)$, are given by inclusion. Moreover, the universal Bass--Serre tree of $\pi_1(\widetilde\caG,\Lambda)$ is isomorphic to $\Gamma_{(\widetilde{\caG},\Lambda)}^\Delta(\caC)$.
\end{cor}
\begin{proof}
    By Theorem~\ref{thm:hpgen}, $\Gamma=\Gamma_{(\caG,\Lambda)}^\Delta(\caC)$ is a tree. 
Note that $g\cdot \Res_T(c)=\Res_T(g\cdot c)$ for all $g\in G$ and $T\subseteq S$. In particular, since~$G$ acts transitively on~$\caC$ with open stabilisers, it acts transitively on $\caR_{S_\lambda}(\caC)$ with open stabilisers~for every $\lambda\in \euV(\Lambda)\sqcup \euE(\Lambda)$. 
    Therefore, $G$ acts on $\Gamma$ with open stabilisers and $G\backslash \Gamma=\Lambda$.
    Let $(\widetilde{\caG},\Lambda)$ be the graph of t.d.l.c.~groups defined in the statement. By the main theorem of Bass--Serre theory, the homomorphism $\varphi\colon \pi_1(\widetilde{\caG},\Lambda)\to G$ induced by inclusion maps $\widetilde{\caG}_\lambda\hookrightarrow G$, $\lambda\in \euV(\Lambda)\sqcup \euE(\Lambda)$, is an isomorphism of abstract groups. Moreover, the universal Bass--Serre tree of $\pi_1(\widetilde\caG,\Lambda)$ is isomorphic to $\Gamma_{(\widetilde{\caG},\Lambda)}^\Delta(\caC)$.
     It remains to prove that $\varphi$ is open and has open inverse. Recall that the t.d.l.c.~group topology of $\pi_1(\widetilde{\caG},\Lambda)$ has a neighbourhood basis at $1$ formed by the open subgroups of the vertex-groups of $(\widetilde{\caG},\Lambda)$. 
Since for every~$v\in \euV(\Lambda)$ the restriction~$\varphi\vert_{\widetilde{\caG}_v}$ is the identity map of~$\widetilde{\caG}_v$, which is an open subgroup of~$G$, ~$\varphi\vert_{\widetilde{\caG}_v}$ is open. Hence, the map~$\varphi$ is open.
   Moreover, since $G$ is a $\sigma$-compact t.d.l.c.~group and $\pi_1(\widetilde{\caG},\Lambda)$ is a t.d.l.c.~group, the open mapping theorem (cf.~\cite[Theorem~6.19]{st:lc}) implies that the continuous bijective homomorphism $\varphi^{-1}\colon G\to \pi_1(\widetilde{\caG},\Lambda)$ is open as well.
\end{proof}
By a repeated application of Theorem~\ref{thm:dav}, M.~W.~Davis~\cite[Proposition~8.8.2]{dav:book} 
proved that every Coxeter group $(W,S)$ admits a visual decomposition $(\caG,\euT)$ (cf.~Definition~\ref{def:visualdec}) over a tree $\caT$ with spherical edge-groups and vertex-groups with at most one end. If~$W$ is infinite and every vertex-group of $(\mathcal G,\euT)$ is a spherical special subgroup of $W$, we say that $(\mathcal G,\euT)$ is a {\em completely spherical visual decomposition of $(W,S)$.}

Combining Theorem \ref{thm:dav}, \cite[Theorem~1.1]{dun:acc}, and~\cite[Theorem~34]{mt:visualdec}, one obtains the following version of the Karrass--Pietrowski--Solitar theorem for Coxeter groups.
\begin{prop}\label{prop:kps}
	Let $(W,S)$ be a Coxeter group. The following are equivalent:
 \begin{itemize}
 \item[(i)] $W$ is infinite and virtually free;
 \item[(ii)] $\vcd(W)=1$ (or, equivalently, $\ccd_\Q(W)=1$);
  \item[(iii)] $(W,S)$ admits a completely spherical visual decomposition;
  \item[(iv)] $(W,S)$ is not spherical, but every clique $\Lambda\subseteq\Gamma_\infty(W,S)$
is spherical, and $\Gamma_\infty(W,S)$ is chordal (cf.~Section~\ref{ss:coxgr}).
   \end{itemize}
\end{prop}

\begin{thm}\label{thm:cd1}
    Let $G$ be a t.d.l.c.~group acting chamber-transitively on a locally finite building $\Delta$ with compact stabilisers. If $\ccd_\Q(G)=1$ then $G$ decomposes as  the fundamental group of a finite tree of profinite groups.
\end{thm}
\begin{proof}
By Theorems~\ref{thm:ihbuil} and~\ref{prop:kps}, $\ccd_\Q(W)=1$ and $W\cong\pi_1(\mathcal{G},\Lambda)$ for some finite tree of spherical special subgroups $(\mathcal{G},\Lambda)$ of $(W,S)$. 
By~Corollary~\ref{cor:G decomp} and since $G$ acts with compact stabilisers on $\Delta$, $G$ acts on $\Gamma_{(\mathcal{G},\Lambda)}^{\Delta}(\caC)$ {with compact open vertex stabilisers} and $G\backslash \Gamma_{(\mathcal{G},\Lambda)}^{\Delta}(\caC)=\Lambda$. 
\end{proof}
\appendix
\section{Coxeter groups of hyperbolic type}
\label{s:hyper}
Let $(W,S)$ be a Coxeter group and put $|S|=n$. Given $V=\spn_{\R}\{\alpha_s\mid s\in S\}$, let $B\colon V\times V\to\R$ be the symmetric bilinear form on $V$ given by
\begin{equation*}
\label{eq:bil}
B(\alpha_s,\alpha_t)=-\cos\Big(\frac{\pi}{m_{st}}\Big),\quad\,\forall\,s,t\in S.
\end{equation*}
Let $\sigma\colon W\to\Gl(V)$ denote the \emph{geometric representation} of~$W$, i.e., the group homomorphism defined by setting 
$$\sigma(s)\colon V\longrightarrow V\qquad \sigma(s)v:=v-2B(\alpha_s,v)\alpha_s$$
for every~$s\in S$; cf.~\cite[\S 5.3]{hum:ref}. The bilinear form $B$ is invariant under $\image(\sigma)$
(cf.~\cite[Proposition~5.3]{hum:ref}). Define the \emph{cone} of $(V,B)$ by
\begin{equation}
\label{eq:cone}
C:=\{\lambda\in V\mid B(\lambda,\alpha_s)>0\ \text{for all $s\in S$}\}.
\end{equation}
The irreducible Coxeter group $(W,S)$ is said to be \emph{of hyperbolic type} if
\begin{itemize}
\item[(a)] $B$ has signature $(n-1,1)$, and
\item[(b)] $B(\lambda,\lambda)<0$ for all $\lambda\in C$.
\end{itemize}
Coxeter groups of hyperbolic type exist only for $3\leq n\leq 10$ (cf.~\cite[p.~141]{hum:ref}). For a Coxeter group of hyperbolic type $(W,S)$ the image of the geometric
representation~$\sigma$ is contained in $O(n-1,1)$. In particular, 
$\image(\sigma)$ is a lattice in $O(n-1,1)$ 
if and only if $(W,S)$ is of hyperbolic type (cf.~\cite[\S6.8, p.~140]{hum:ref} and the references therein).
If $(W,S)$ is of hyperbolic type and crystallographic as a Coxeter group, then $\image(\sigma)$ also
stabilises some lattice in~$V$ (cf.~\cite[\S6.6, p.~135]{hum:ref}). 
\begin{fact}\label{fact:lattice} Let $(W,S)$ be a crystallographic hyperbolic Coxeter group. There exists a  $\sigma(W)$-invariant lattice $\Lambda\subset V$ that is freely generated by a set
$\{\lambda_s\mid s\in S\}$  such that
$B(\lambda_i,\lambda_j)\in\tfrac{1}{e}\Z$ for some~$e\in\Z_{\geq 1}$.
\end{fact}
\begin{proof}
    Let $\Lambda\subset V$ be the $\Z$-lattice generated by 
    $\{\lambda_s=c_s\alpha_s\mid s\in S\}$, where the~$c_s$'s are defined as described in~\cite[\S 6.6, p.~136]{hum:ref}. 
    Define $d\colon\Lambda\times\Lambda\to\Q$ by~$d(x,y)=\frac{2B(x,y)}{B(y,y)}$. Then $d(\lambda_s,\lambda_t)\in\Z$ for all $s,t\in S$ (cf.~\cite[\S 6.6, p.~136]{hum:ref}). Hence, for every $s,t\in S$, $2B(\lambda_s,\lambda_t)\in\tfrac{1}{f}\Z$, where $f=\mathrm{lcm}\{B(\lambda_s,\lambda_s)\mid s\in S\}$.
\end{proof}
Let $\Lambda$ be the lattice from Fact~\ref{fact:lattice}. 
Defining $$\boG(\Z)=\{\alpha\in\Gl_\Z(\Lambda)\mid B(\alpha(v_i),\alpha(v_j))=B(v_i,v_j)\}$$
yields an affine algebraic group scheme $\boG$ defined over $\Z$ that satisfies
\begin{equation}
\label{eq:arith}
W\subseteq \boG(\Z)\subset\boG(\R)=O(n-1,1).
\end{equation}
Since $W$ is a lattice in $O(n-1,1)$, so is $\boG(\Z)$. Moreover,
$|\boG(\Z)\colon W|<\infty$. Consequently,
\begin{fact}
\label{fact:crhypcox}
Let $(W,S)$ be a crystallographic Coxeter group of hyperbolic type. Then~$W$ is an arithmetic lattice in~$O(n-1,1)$, where~$n=|S|$.
\end{fact}
The Coxeter group~$(W,S)$ of hyperbolic type is said to be {\em of compact hyperbolic type} if $W$ is cocompact in $O(n-1,1)$ (cf.~\cite[\S 6.9]{hum:ref}). Otherwise, $(W,S)$ is said to be \emph{of non-compact hyperbolic type}. 
\begin{prop}\label{prop:vcdhyp}
    Let $(W,S)$ be a Coxeter group of compact hyperbolic type. Then $$\algrk(W)=1\quad\text{and}\quad\ccd_\Q(W)=\vcd(W)=|S|-1\geq 2.$$
\end{prop}
\begin{proof}
     It is well-known that every cocompact discrete subgroup
 of $O(n-1,1)$ is Gromov-hyperbolic (cf.~\cite[p.~93]{grom:hyp}). Hence, $\algrk(W)=1$. 

 By~\cite[\S 6.8]{hum:ref}, every proper special subgroup of $(W,S)$ is spherical. Let $(W^\prime,S^\prime)$ be the affine Coxeter group of type $\tilde{A}_n$, where $n=|S|-1$. The poset of all spherical special subgroups of $(W,S)$ is isomorphic to the poset of all spherical special subgroups of~$(W',S')$ since~$|S'|=|S|$. 
 Therefore, for $R\in\{\Z,\Q\}$, $B_R(W,S)$ and $B_R(W',S')$ have the same dimension. By~\cite[p.~21, Theorem]{bes:vcd}, one has
\begin{equation*}
    \ccd_\Q(W)=\ccd_\Q(W')=|S|-1=\vcd(W')=\vcd(W).\qedhere
\end{equation*}
\end{proof}
\begin{rem} To deduce that~$\algrk(W)=1$ for every Coxeter group~$(W,S)$ of compact hyperbolic type, one may also argue as follows.
By~\cite[\S 6.8]{hum:ref}, every proper special subgroup of $(W,S)$ is spherical. Hence, the technical condition~(iii) of~\cite[Theorem~17.1]{mouss} is satisfied. In turn, the latter is equivalent to the condition~(ii) of~\cite[Theorem~17.1]{mouss}, which states that~$W$ has no subgroup isomorphic to~$\Z^2$.
\end{rem}

\begin{example}\label{ex:comphyp}
    From {Lann\'er classification} (cf.~\cite[\S 6.9]{hum:ref}),
\cite[Proposition~6.6]{hum:ref} and \cite[\S 6.7]{hum:ref}, 
the Coxeter group associated to the Coxeter diagram
\begin{equation*}
\label{eq:comp1}
\xymatrix{\bullet\ar@{-}[d]\ar@{-}[r]^4&\bullet\ar@{-}[d]\\
\bullet\ar@{-}[r]^4&\bullet}
\end{equation*}
is the only crystallographic Coxeter group of compact hyperbolic type
of rank greater than or equal to $4$. 
 Therefore, by Proposition~\ref{prop:vcdhyp} and Proposition~\ref{prop:algcd}, if $G$ is a group with a root group datum of type $(W,S)$ and finite root groups,
  the geometric completion $\bar G$ of $G$ satisfies:
  $$1=\algrk(W)=\flatrk(\bar G)\quad\text{and}\quad\ccd_{\Q}(\bar G)=\ccd_\Q(W)=\vcd(W)=3.$$
\end{example}
 \begin{prop}
 \label{prop:noncomp}
Let $(W,S)$ be a Coxeter group of non-compact hyperbolic type.  
Then
\begin{equation*}
\label{eq:ncom}
\vcd(W)=\algrk(W)= |S|-2.
\end{equation*}
In particular, $\ccd_\Q(W)=|S|-2$.
 \end{prop}
 
 \begin{proof}
 Since $\algrk(W)\leq \ccd_\Q(W)\leq \vcd(W)$ (cf.~Lemma~\ref{lem:algvcd}), it suffices to prove the first part of the statement.
 Throughout the proof, put $|S|=n$.
 By~\cite[p.~140]{hum:ref}, every Coxeter group of non-compact hyperbolic type $(W,S)$ 
contains an affine special subgroup $(W^\prime,S^\prime)$ with $|S^\prime|=n-1$.
Hence, 
\begin{equation}
\label{ineq:ncom1}
\vcd(W)\geq\algrk(W)\geq\algrk(W')=n-2.
\end{equation}
It remains to show that $\vcd(W)\leq n-2$. We first deal with the  crystallographic case. Let $X$ denote the symmetric space associated to $O(n-1,1)$. Then,
\begin{equation}
\label{eq:dimX}
 \dim(X)=\binom{n}{2}-\binom{n-1}{2}=n-1.
 \end{equation}
Since $W\backslash X$ is non-compact, one has $\partial X\neq\emptyset$ (cf.~\cite[Theorem~9.3]{borser:cor}). Equivalently, if $\boG$ denotes the affine algebraic group scheme defined over $\Z$ satisfying~\eqref{eq:arith}, the $\Q$-rank~$\ell(\boG)$ of the group scheme~$\boG$ 
 satisfies
 $\ell(\boG)\geq 1$. This yields $$\vcd(W)=\dim(X)-\ell(\boG)\leq n-2,$$ (cf.~\cite[Corollary~11.4.3]{borser:cor}).

In the remaining non-crystallographic cases (cf.~\cite[\S 6.9, Figure~3]{hum:ref}), 
 Proposition~\ref{fact:ncsph} below applies and the Bestvina complex $B_\Z(W,S)$ has at most $n-1$ vertices. By~\cite{bes:vcd}, one concludes that $\vcd(W)\leq n-2$.
\end{proof}
\begin{prop}\label{fact:ncsph}
    Let $(W,S)$ be a Coxeter group of non-compact hyperbolic type which is not crystallographic. Then $(W,S)$ admits exactly $|S|-1$ maximal spherical subsets.
    \end{prop}
\begin{proof}
In what follows, by an \emph{$r$-subset} of~$S$ we mean a subset with $r$~elements. Moreover, we make use of the classifications of Coxeter groups of spherical, affine and hyperbolic types by their Coxeter diagram as given, for example, in~\cite[pp.~32, 34, 142-144]{hum:ref}). Since $(W,S)$~is a Coxeter group of hyperbolic type, observe that every $r$-subset of~$S$, with $r\leq |S|-2$, is spherical (cf.~\cite[Proposition~6.8]{hum:ref}, recalling that every proper special subgroup of a Coxeter group of positive type is spherical). Thus the maximal spherical subsets of $S$ have cardinality either $|S|-1$ or $|S|-2$.

Let first $|S|=3$, say $S=\{a,b,c\}$. 
Since $(W,S)$ is non-compact hyperbolic, at least one edge of $\Gamma(W,S)$ is labelled by~$\infty$, say $m_{ab}=\infty$. Moreover, since~$(W,S)$ is not crystallographic, either $m_{ac}<\infty$ or $m_{bc}<\infty$. If both~$m_{ac}$ and~$m_{bc}$ are finite, then the maximal spherical subsets of~$S$ are~$\{a,c\}$ and~$\{b,c\}$. If $m_{bc}=\infty$ and~$m_{ac}<\infty$, then the maximal spherical subsets are~$\{a,c\}$ and~$\{b\}$. An analogous conclusion holds if $m_{ac}=\infty$ and $m_{bc}<\infty$.
     
Assume now that $|S|\geq 4$. By~\cite[\S 6.9, Figure~3]{hum:ref}, there are only seven cases: six with $|S|=4$ and one with $|S|=6$.
     If $S=\{a,b,c,d\}$, then $(W,S)$ has one of following Coxeter diagrams:
     \begin{equation}\label{eq:S=4.1}
         \xymatrix{
         {^d{\bullet}}\ar@{-}[r]\ar@{-}[d]_4 & {{\bullet}^c}\ar@{-}[d]^4\\
         {_{a}{\bullet}}\ar@{-}[r]_4           & {{\bullet}_{b}}
         } \qquad 
         \xymatrix{
         {^d{\bullet}}\ar@{-}[r]\ar@{-}[d] & {{\bullet}^c}\ar@{-}[d]\\
         {_{a}{\bullet}}\ar@{-}[r]_6           & {{\bullet}_{b}}
         } \qquad
         \xymatrix{
         {^d{\bullet}}\ar@{-}[r]^4\ar@{-}[d] & {{\bullet}^c}\ar@{-}[d]\\
         {_{a}{\bullet}}\ar@{-}[r]_6           & {{\bullet}_{b}}
         } \qquad
         \xymatrix{
         {^d{\bullet}}\ar@{-}[r]^5\ar@{-}[d] & {{\bullet}^c}\ar@{-}[d]\\
         {_{a}{\bullet}}\ar@{-}[r]_6           & {{\bullet}_{b}}
         }
         \end{equation}
         \begin{equation}\label{eq:S=4.2}
         \xymatrix@R=0,3cm{
                        &        & {\bullet}^b\ar@{-}[dd]\\
        {\stackrel{d}{\bullet}}\ar@{-}[r]_5 &\stackrel{a}{\bullet}\ar@{-}[ur]\ar@{-}[dr] & \\
                        &         & {\bullet}_c
         }
          \qquad \qquad
         \xymatrix{
         \stackrel{a}{\bullet}\ar@{-}[r]_6 & \stackrel{b}{\bullet}\ar@{-}[r] & \stackrel{c}{\bullet}\ar@{-}[r]_5 & \stackrel{d}{\bullet}
         }
     \end{equation}
     Let first $(W,S)$ be associated to one of the Coxeter diagrams in \eqref{eq:S=4.1}. Among the $3$-subsets of $S$, only $\{a,c,d\}$ and $\{b,c,d\}$ are spherical. Note also that  $\{a,b\}$ is the only $2$-subset of $S$ which is not contained in both $\{a,c,d\}$ and $\{b,c,d\}$. Thus, $\{a,c,d\}$ and $\{b,c,d\}$ and $\{a,b\}$ are the maximal spherical subsets of $S$.
     Assume now that $(W,S)$ is associated to one of the Coxeter diagrams in \eqref{eq:S=4.2}. Among the $3$-subsets of $S$, the following ones are spherical: $\{a,b,d\}$, $\{a,c,d\}$ and $\{b,c,d\}$. Hence, they are the maximal spherical subsets of $S$.

     Finally, let $S=\{a,b,c,d,e,f\}$. Then $(W,S)$ has the following Coxeter diagram:
     \begin{equation*}\label{eq:S=6}
         \xymatrix{
                                  & {\stackrel{a}{\bullet}}\ar@{-}[rr] &  & \stackrel{b}{\bullet} \ar@{-}[dr] &              \\
    {_f\bullet}\ar@{-}[ur]\ar@{-}[dr] &          &          &                     
                                  & {\bullet_c}\ar@{-}[dl] \\
                                  & {_e\bullet}\ar@{-}[rr]_4 &  & {\bullet_d}  &      &
         }
     \end{equation*}
    For $x\in S$, the set $S\setminus \{x\}$ is spherical if and only if $x\not\in\{a,b\}$. In particular, the $4$-subset $S\setminus\{x,y\}$ is not contained in any spherical $5$-subset of $S$ if and only if $\{x,y\}=\{a,b\}$. One concludes that the maximal spherical subsets of $S$ are $S\setminus \{x\}$, for some $x\not\in\{a,b\}$, and $S\setminus \{a,b\}$. 
\end{proof}
\begin{rem}\label{rem:claimcry}
     Proposition~\ref{fact:ncsph} is \emph{not} true for arbitrary crystallographic Coxeter groups of non-compact hyperbolic type.
     For example, let $(W,S)$ be the Coxeter group defined by the following Coxeter diagram:
     \begin{equation}\label{eq:diag}
         \xymatrix@R=0,3cm{
                        &        & {\bullet}^c\ar@{-}[dd]\\
        {\stackrel{a}{\bullet}}\ar@{-}[r]_6 &\stackrel{b}{\bullet}\ar@{-}[ur]\ar@{-}[dr] & \\
                        &         & {\bullet}_d
         }
     \end{equation}
     Among the $3$-subsets of $S=\{a,b,c,d\}$, only $\{a,c,d\}$ is spherical. Moreover, there are three spherical $2$-subsets of $S$, namely $\{a,b\}$, $\{b,c\}$ and $\{b,d\}$, that are not contained in $\{a,c,d\}$. As a conclusion, $(W,S)$ admits exactly $4=|S|$ maximal spherical subsets and so does not satisfy the conclusion of Proposition~\ref{fact:ncsph}.

Note that for such a Coxeter group one can easily draw the Bestvina complex, which is indeed 2-dimensional (see Figure~\ref{fig:best}).
 \end{rem}
 \begin{figure}[H]\label{fig:best}
 \centering
\begin{tikzpicture}[scale=1.5]
    \coordinate (A) at (0,0);
    \coordinate (B) at (4,0);
    \coordinate (C) at (2,3);
    \coordinate (P) at (2,1); 

    \fill[fill=black!10] (A) -- (B) -- (P) -- cycle;
    \fill[fill=black!10] (B) -- (C) -- (P) -- cycle;
    \fill[fill=black!10] (C) -- (A) -- (P) -- cycle;

    \draw[line width=1pt] (A) -- (B) -- (C) -- cycle;

    \draw[line width=1pt] (P) -- (A);
    \draw[line width=1pt] (P) -- (B);
    \draw[line width=1pt] (P) -- (C);

    \filldraw (A) circle (1.5pt) node[below left] at (A) {$\{a,b\}$};
    \filldraw (B) circle (1.5pt) node[below right] at (B) {$\{b,c\}$};
    \filldraw (C) circle (1.5pt) node[above] at (C) {$\{b,d\}$};
    \filldraw (P) circle (1.5pt) node[above right] {$\{a,c,d\}$};
\end{tikzpicture}
\caption{The Bestvina complex $B_\Z(W,S)$ associated to the Coxeter group~\eqref{eq:diag} with vertices $\{a,c,d\},\{a,b\},\{b,d\}$ and $\{b,c\}$.}
\end{figure}

\section{Affine Coxeter groups and Coxeter groups of rank 3}\label{app:maxrk}
For every Coxeter group $(W,S)$, recall from Lemma~\ref{lem:algvcd} that
\begin{equation}\label{eq:algcd}
    \algrk(W)\leq\ccd_\Q(W)\leq\vcd(W)\leq |S|-1.
\end{equation}
The maximum possible algebraic rank for~$(W,S)$, namely $|S|-1$, is attained only in the affine case as shown below.
\begin{prop}\label{prop:maxalgrk}
   Let $(W,S)$ be a  Coxeter group with $|S|\geq 2$. Then $(W,S)$ is affine if and only if  $\algrk(W)=|S|-1$. Consequently, if $(W,S)$ is an affine Coxeter group then $\algrk(W)=\ccd_\Q(W)=\vcd(W)=|S|-1.$
\end{prop}
\begin{proof} Assume that $(W,S)$ is affine. Then $W$ is the semidirect product of its translation subgroup $T$ by a spherical Coxeter group (cf.~\cite[Proposition~4.2]{hum:ref}). In particular, $T$ is a free abelian subgroup of rank $|S|-1$ and $\algrk(W)=|S|-1$ (cf.~\cite[Proposition~2.3]{MV24}).

Now assume that $\algrk(W)=|S|-1$. Since $|S|-1>0$, $(W,S)$ is non-spherical. 
By \cite[Theorem~6.8.3]{krammer}, 
the algebraic rank of $W$ coincides with the rank of a so-called standard free abelian subgroup, say $\prod_{j=1}^kH_{j}$ (cf.~\cite[p.~154]{krammer}). In particular, there are irreducible non-spherical subsets $I_1, \dots, I_k$ of $S$ which are pairwise perpendicular (i.e., for all $j\neq l$, every element of $I_{j}$ commutes with every element of $I_l$) and such that
 \begin{equation}\label{eq:algrk<=}
        |S|-1=\algrk(W)=\sum_{j=1}^k\algrk(H_{j})\leq \sum_{j=1}^k(|I_j|-1).
    \end{equation}
    As the $I_j$'s are pairwise disjoint, $|I_1|+\dots+|I_k|\leq |S|$ and, by \eqref{eq:algrk<=}, one has 
    $$|S|-1\leq \sum_{j=1}^k |I_j|-k\leq |S|-k.$$
    Necessarily, $k=1$ and $I_1=S$. Since $|S|\geq 2$, the algebraic rank of $H_1$ is greater than $1$ and then $H_1$ is the translation subgroup of the affine special subgroup $(W_{I_1}, I_1)$ (cf.~\cite[p.~154]{krammer}). One concludes that $(W,S)=(W_{I_1}, I_1)$ is affine.

   The last part of the statement is a consequence of the inequality~\eqref{eq:algcd}.
\end{proof}
 Given a Coxeter group $(W,S)$ of rank $3$, we completely describe how the virtual cohomological dimension and the rational cohomological dimension of $W$ relate to its algebraic rank.
\begin{prop}\label{prop:3gen}
    Let $(W,S)$ be a Coxeter group with $S=\{a,b,c\}$, $|S|=3$ and with Coxeter matrix~$[m_{st}]_{s,t\in S}$. Then $(W,S)$ falls into one of the following cases:
    \begin{itemize}
       \item[(i)] $\algrk(W)=\ccd_\Q(W)=\vcd(W)=0$ and  $(W,S)$ is spherical, i.e., $(W,S)$ has of one of the following types: $A_1^3$, {$A_3$, $B_3$,} $H_3$ or $I_2(m)\times A_1$ where $3\leq m<\infty$;
        \item[(ii)] $\algrk(W)=\ccd_\Q(W)=\vcd(W)=2$ and $(W,S)$ is affine, i.e., $(W,S)$ has of one of the following types: $\widetilde{A}_2$, $\widetilde{B}_2=\widetilde{C}_2$ or $\widetilde{G}_2$;
        \item[(iii)] $\algrk(W)=\ccd_\Q(W)=\vcd(W)=1$. This happens if and only if at least one among $m_{ab}$, $m_{bc}$ and $m_{ca}$ is $\infty$;
       \item[(iv)] $\algrk(W)=1<\ccd_\Q(W)=\vcd(W)=2$. Then~$(W,S)$ is compact hyperbolic of rank 3.
    \end{itemize}
\end{prop}

\begin{proof}
  By \eqref{eq:algcd}, one has $\algrk(W)\leq \ccd_\Q(W)\leq\vcd(W)\leq 2$.
  By~\cite[Corollary~5.4]{hum:ref}, every Coxeter group is a finitely generated linear group over~$\R$.
 Hence, by Tits alternative (cf.~\cite[Theorem 1]{tits}), any torsion Coxeter group must be finite. Thus
$\algrk(W)=0$ is equivalent to $(W,S)$ being spherical. Moreover,
$\algrk(W)=2$ is equivalent to $(W,S)$ being affine (cf.~Proposition~\ref{prop:maxalgrk}). This yields (i) and~(ii). 

Hence,
we may assume that~$\algrk(W)=1$ and
$\vcd(W)\in\{1,2\}$. By Stallings--Swan theorem (cf.~\cite[p.~1, Corollary]{swan}), $\vcd(W)\leq 1$ if and only if $W$~is virtually free. Since~$\vcd(W)=0$ if and only if~$W$ is finite, one deduces that~$\vcd(W)= 1$ if and only if~$W$ is infinite and virtually free. In turn, 
by~\cite[Theorem~34]{mt:visualdec}, $W$ is infinite and virtually free if and only if
$\Gamma_\infty(W,S)$~is chordal and every clique in~$\Gamma_\infty(W,S)$ is spherical. Since $|S|=3$, the latter conditions on the presentation graph are equivalent to require that at least one among $m_{ab}$, $m_{bc}$ and $m_{ca}$ is $\infty$. Thus~(iii) holds.

It remains to investigate the case $\algrk(W)=1$ and $\vcd(W)=2$. As $(W,S)$ is neither spherical nor affine,
the bilinear form 
$B\colon V\times V\to\R$ is non-degenerate and has signature $(2,1)$ (cf.~\cite[p.~138]{hum:ref}).
Moreover, as $|S|=3$ and none of the numbers
$m_{ab}$, $m_{bc}$ or $m_{ca}$ can be ~$\infty$, all proper special subgroups of $(W,S)$ must be spherical. Thus $(W,S)$ must be compact hyperbolic (cf.~\cite[p.~140]{hum:ref}).
\end{proof} 

\begin{example}\label{ex:kacmoody}
    Let $S=\{s_1,s_2,s_3\}$. The Coxeter groups associated, respectively, to the Coxeter diagrams
\begin{equation}
\label{eq:comp2}
\xymatrix{\stackrel{s_1}{\bullet}\ar@{-}[r]^6&\stackrel{s_2}{\bullet}\ar@{-}[r]^6&\stackrel{s_3}{\bullet}}
\qquad\qquad
\xymatrix@R=0,3cm{&\stackrel{s_1}{\bullet}\ar@{-}[2,0]\\
\stackrel{s_2}{\bullet}\ar@{-}[ur]^4\ar@{-}[dr]_4&\\
&\stackrel{s_3}{\bullet}}\qquad\qquad
\xymatrix@R=0,3cm{&\stackrel{s_1}{\bullet}\ar@{-}[2,0]\\
\stackrel{s_2}{\bullet}\ar@{-}[ur]^6\ar@{-}[dr]_6&\\
&\stackrel{s_3}{\bullet}}
\end{equation}
are the only crystallographic Coxeter groups of hyperbolic type of rank~$3$. 

Let $(W,S)$ be the Coxeter group associated to one of the previous Coxeter diagrams. Since $(W,S)$ is crystallographic, it is possible to construct a generalised Cartan matrix\footnote{A \emph{generalised Cartan matrix} over a finite non-empty set~$I$ is any matrix~$A=[A_{ij}]_{i,j\in I}$ with integral entries such that~$A_{ii}=2$ for every~$i\in I$, $A_{ij}\leq 0$ for all~$i\neq j$, and~$A_{ij}=0\Leftrightarrow A_{ji}=0$.}~$A=[A_{ij}]_{1\leq i,j\leq 3}$ such that the associated Coxeter group is~$(W_A,S)$. 
Indeed, for all~$i,j\in\{1,2,3\}$ with $i\neq j$, the integral entries of the matrix~$A$ must be chosen accordingly to the following table  (cf.~\cite[Proposition~4.22]{marq:km}):
\begin{equation*}
    \begin{tabular}{c|ccccc}
        $A_{ij}A_{ji}$ & $0$ & $1$ & $2$ & $3$ & $\geq 4$  \\\hline
        $m_{ij}$ & $2$ & $3$ & $4$ & $6$ & $+\infty$
    \end{tabular}
\end{equation*}
where the~$m_{ij}=m_{s_i,s_j}$ are the entries of the Coxeter matrix of~$(W_A,S)$. For example, the generalised Cartan matrices
\begin{equation*}
    A_1=\begin{bmatrix}
        \phantom{-}2 & -3 & \phantom{-}0\\
        -1 & \phantom{-}2 & -3\\
        \phantom{-}0 & -1 & \phantom{-}2
    \end{bmatrix}, \, A_2=\begin{bmatrix}
        \phantom{-}2 & -2 & -1\\
        -1 & \phantom{-}2 & -2\\
        -1 & -1 & \phantom{-}2
    \end{bmatrix}, \,A_3=\begin{bmatrix}
        \phantom{-}2 & -3 & -1\\
        -1 & \phantom{-}2 & -3\\
        -1 & -1 & \phantom{-}2
    \end{bmatrix}
\end{equation*}
are, respectively, associated to the Coxeter diagrams in~\eqref{eq:comp2}.

Hence, chosen any finite field $\mathbb F$, the complete Kac--Moody group~$\overline{G_i(\mathbb F)}$ with associated generalised Cartan matrix~$A_i$ provides an example of a t.d.l.c.~group satisfying~$\flatrk(\overline{G_i(\mathbb F)})=1$ and $\ccd_{\Q}(\overline{ G_i(\mathbb F)})=2$ (cf.~Proposition~\ref{prop:3gen}(iv), the paragraph preceding~Proposition~\ref{prop:algcd} and Theorem~\ref{thm:ihbuil}). By Proposition~\ref{prop:3gen},
these provide the ``smallest'' examples of complete Kac--Moody groups exhibiting a discrepancy between their flat rank and their rational discrete cohomological dimension.
\end{example}
Propositions \ref{prop:maxalgrk} and \ref{prop:3gen}, together with the results in Appendix~\ref{s:hyper}, provide a complete description of the invariants $\ccd_\Q(\_)$ and $\vcd(\_)$ for the following classes of Coxeter groups: affine,  hyperbolic and rank~$3$. In all these cases the rational cohomological dimension coincides with the virtual cohomological dimension. Since it is well-known that there are Coxeter groups satisfying $\vcd(W)\neq\ccd_\Q(W)$ (cf.~\cite[Remarks~(3)]{bes:vcd}), we are led to the following question:
\begin{ques}\label{ques:algvcd}
    For which Coxeter groups $(W,S)$ does $\ccd_\Q(W)=\vcd(W)$ hold?
\end{ques}
 In the next section, we examine a final case where we are only able to establish a bound for the rational cohomological dimension.
\section{Tree decompositions and cohomological dimension}\label{S:cd_coxeter}
Recall that every  undirected graph
$\Gamma=(V,E)$ defines a combinatorial graph $\Vec{\Gamma}=(V,\Vec{E})$ in the sense of J-P.~Serre (cf.~Section~\ref{ss:naiGr}). In the following, we will deal only with undirected graphs and refer to them as “graphs”. 

\smallskip
Let $\Gamma=(V,E)$ be a graph. A \emph{tree decomposition} of~$\Gamma$ is a pair $(\caT,\caV)$ that consists of a tree~$\euT$ and a family~$\caV=\{V_t\}_{t\in V(\euT)}$ of non-empty subsets of~$V$ labelled by the vertex set $V(\euT)$ of~$\euT$ (with the convention that $V=\emptyset$ implies $V(\euT)=\emptyset$), and that satisfies the following properties:
\begin{itemize}
    \item[(T1)] $V=\bigcup_{t\in V(\euT)}V_t$;
    \item[(T2)] for every $\{x,y\}\in E$, there is $t\in V(\euT)$ such that $x,y\in V_t$; 
    \item[(T3)] for all $s,t\in V(\euT)$ and every $u\in V(\euT)$ occurring in the geodesic in $\euT$ connecting $s$ and $t$, one has $V_s\cap V_t\subseteq V_u$.
\end{itemize}
A tree decomposition $(\euT,\euV)$ is said to be \emph{reduced} if $E(\euT)=\emptyset$ or, for all $\{t,u\}\in E(\euT)$, $V_t\cap V_u$ is properly contained in both $V_t$ and $V_u$.
\begin{rem}
    The notion of tree decomposition was originally introduced by R.~Halin, although it became an established tool in graph theory only after the series of articles by N.~Robertson and P.~D.~Seymour (cf.~the introduction of~\cite{roseIII} for instance).
\end{rem}
\begin{example}[\protect{cf.~\cite[Example~1.1]{dunkr:vercut}}]\label{ex:treedec1}
 Let $\Gamma$ be a connected graph, and denote by $\big(\Lambda_i=(V(\Lambda_i),E(\Lambda_i))\big)_{i\in I}$ the family of all its maximal $2$-connected subgraphs. Recall that a graph is \emph{$2$-connected} if it is connected and removing an arbitrary vertex from the graph keeps the graph connected.
Moreover, let $(v_j)_{j\in J}$ be the family of all \emph{cut-vertices} of $\Gamma$, i.e., all those vertices $v_j$ for which $\Gamma-v_j$ is no longer connected. A \emph{block} of $\Gamma$ is either a maximal $2$-connected subgraph or a cut-vertex of $\Gamma$.
The \emph{block decomposition} of $\Gamma$ is the pair $(\caT,\caV)$ consisting of a tree $\caT=(V(\caT),E(\caT))$ with bipartition $V(\euT)=I\sqcup J$, $E(\euT)=\{\{i,j\}\in \caP_2(V(\caT))\mid i\in I,\,j\in J,\,v_j\in V(\Lambda_i)\}$, and a collection of subsets of $V$, say $\caV=\{V_k\}_{k\in I\sqcup J}$, defined by $V_i=V(\Lambda_i)$ for every $i\in I$, and by $V_j=\{v_j\}$ for every $j\in J$. 
One proves that $(\caT,\caV)$ is a tree decomposition of~$\Gamma$. 
\end{example}
\begin{example}\label{ex:treedec2}
  Let $\Gamma=(V,E)$ be a finite undirected graph. The \emph{graph of maximal cliques 
$\mx(\Gamma)=(V(\mx(\Gamma)),E(\mx(\Gamma)))$} is given by
\begin{equation}
\label{eq:maxcl}
\begin{aligned}
V(\mx(\Gamma))&=\{\Lambda\subseteq\Gamma\mid
\Lambda\ \text{a maximal clique}\}\\
E(\mx(\Gamma))&=\big\{\{\Lambda,\Xi\}\subseteq
\caP_2(V(\mx(\Gamma)))\mid\ 
V(\Lambda)\cap V(\Xi)\neq\emptyset\big\}.
\end{aligned}
\end{equation}
A maximal subtree $\caT\subseteq{\mx}(\Gamma)$ has the \emph{clique intersection property} if, for every pair $(\Lambda,\Xi)$ of maximal cliques in $\Gamma$, the unique geodetic in $\caT$, determined by the vertex-sequence $(\Lambda=\Xi_0, \Xi_1,\dots,\Xi_{r}=\Xi)$, satisfies $\Lambda\cap\Xi\subseteq\Xi_i$ for all $0\leq i\leq r$.
 The existence of a maximal subtree $\caT\subseteq{\mx}(\Gamma)$ satisfying this property is equivalent to requiring that $\Gamma$ is chordal (cf.~\cite[Theorem~3.2]{blpey} and~\cite[Proposition~5.5.1]{diest}).
Provided $\Gamma$ is chordal, the pair $(\caT,\{V_\Lambda\}_{\Lambda\in V(\euT)})$, where $V_\Lambda\subseteq V$ is the set of vertices of the clique $\Lambda$ in $\Gamma$, is a tree decomposition of $\Gamma$.
\end{example}
Using tree decompositions of the presentation diagram of a Coxeter group $(W,S)$, one obtains the following abstract decomposition theorem. 
\begin{thm}\label{thm:abstDec}
    Let~$(W,S)$ be a Coxeter group and let $(\euT,\caV)$ be a tree decomposition of the presentation diagram~$\Gamma_\infty(W,S)$. Consider the tree of special subgroups~$(\caG,\vec{\euT})$ of $(W,S)$ given by setting $\caG_t=W_{V_t}$ for every $t\in V(\vec{\euT})$, and $\caG_{\{s,t\}}=W_{V_s\cap V_t}$ for every $\{s,t\}\in E(\vec{\euT})$. Then the inclusion maps $\caG_t\hookrightarrow W$, $t\in V(\vec{\euT})$, induce a group isomorphism
    $$W\cong \pi_1(\caG,\vec{\euT}).$$
\end{thm}
\begin{proof}
    For every $t\in V(\vec{\euT})$, denote by $\langle V_t\mid \caR_t\rangle$ the Coxeter presentation of $(W_{V_t},V_t)$ (i.e., $\caR_t=\{(st)^{m_{st}}\mid s,t\in V_t\text{ and }m_{st}<\infty\}$). Then $\pi_1(\caG,\vec{\euT})$ admits the following presentation:
    \begin{equation}\label{eq:prespi}
    \Bigg\langle \bigsqcup_{t\in V(\vec{\euT})} V_t\,\Bigg|\,
     \bigsqcup_{t\in V(\vec{\euT})}\caR_{t};\,\,v_s=v_t,\,\forall\,v\in V_s\cap V_t,\,\forall\,\{s,t\}\in E(\vec{\euT})\Bigg\rangle,
\end{equation}
where $v_s$ and $v_t$ denote the copies of the vertex $v\in V_s\cap V_t$ in the sets $V_s$ and $V_t$, respectively.
Let $\langle S\mid \caR\rangle$ be the Coxeter presentation of~$(W,S)$.
Since~$S$ is a minimal set of generators for~$W$ (see~\cite[Corollary~1.4.8]{bb:cox}), one has $S=\bigcup_{t\in V(\vec{\euT})}V_t$ and hence, by~(T2), $\caR=\bigcup_{t\in V(\vec{\euT})}\caR_t$. It remains to prove that the two disjoint unions appearing in~\eqref{eq:prespi} can be relaxed to unions (i.e., $\pi_1(\caG,\vec{\euT})$ admits the presentation $\langle S\mid \caR\rangle$). To this end, for every $\{s,t\}\in E(\vec{\euT})$, denote by~$\frp_{s,t}$ the unique path without backtrackings in~$\euT$ connecting $s$ to $t$. For every $e\in E(\vec{\euT})$, denote by $\iota_e$ the inclusion map $\caG_e\hookrightarrow \caG_{t(e)}$. Because of~(T3), for all edges $e=\{u,v\}$ and $f=\{w,x\}$ occurring in $\frp_{s,t}$, the images $\iota_e(V_u\cap V_v)$ and $\iota_f(V_w\cap V_x)$ are identified pointwise. 
\end{proof}
\begin{defi}\label{defn:complP}
    For a Coxeter group~$(W,S)$ one says that the presentation diagram~$\Gamma_\infty(W,S)$ has the \emph{$\caP$-clique property} if, for every
clique $(S^\prime,\caP_2(S^\prime))$ of $\Gamma_\infty(W,S)$, the Coxeter group
$(W_{S'},S^\prime)$
has property $\caP$.

Moreover, a visual decomposition~$(\caG,\caT)$ of~$(W,S)$ (cf.~Definition~\ref{def:visualdec}) is said to be \emph{completely-$\caP$} if $W\cong\pi_1(\caG,\caT)$ for some tree of special subgroups~$(\caG,\caT)$ such that all vertex-groups and edge-groups satisfy property~$\caP$.
\end{defi}
Note that the completely-$\caP$ visual decompositions with $\caP$ being “finite” are precisely the completely spherical visual decompositions  defined in the paragraph preceding Proposition~\ref{prop:kps}.
 
Combining Example~\ref{ex:treedec2} and Theorem~\ref{thm:abstDec}, one deduces the following.
\begin{cor}\label{cor:chord}
Let $(W,S)$ be an infinite Coxeter group such that $\Gamma_\infty(W,S)$ has the $\caP$-clique property.
If $\Gamma_\infty(W,S)$ is chordal, then $(W,S)$ admits a completely-$\caP$ visual decomposition $\pi_1(\caG,\caT)$. 
 In particular, $\caT$ can be chosen to be a maximal subtree of the Serre graph $\overrightarrow{\mx}(\Gamma_\infty(W,S))$ associated to~$\mx(\Gamma_\infty(W,S))$ (see Section~\ref{ss:naiGr}) satisfying the clique intersection property.
\end{cor}
Being finite can be considered as a property $\caP$. In this case, one can use Corollary~\ref{cor:chord} to give an alternative proof of the implication (iv)$\Rightarrow$(iii) in Proposition~\ref{prop:kps}. 
Another example of a suitable property $\caP$ is the following: a group $G$ is said to be {\em slender} (or {\em Noetherian}) if every subgroup of $G$ is finitely generated. 
Slender Coxeter groups are characterised in~\cite[Proposition~C]{varg}.
By Corollary~\ref{cor:chord}, one deduces the following result due to O.~Varghese (cf.~\cite[Corollary~3.5]{varg}).
\begin{cor} Let $(W,S)$ be a Coxeter group. If $\Gamma_\infty(W,S)$ is chordal and has the slender-clique property, then $(W,S)$ is coherent (i.e., every finitely generated subgroup is finitely presented).
\end{cor}
\begin{proof}
By Corollary~\ref{cor:chord}, $(W,S)$ has a completely-slender visual decomposition. By~\cite[Corollary~3.5]{varg}, every slender Coxeter group is coherent. The claim follows by the fact that the free product of coherent groups amalgamated over a slender subgroup is coherent (cf.~\cite[Theorem~8]{ks70}).
\end{proof}
\smallskip

Another application of Theorem~\ref{thm:abstDec} involves the rational cohomological dimension of a Coxeter group, as stated in the following proposition.
\begin{prop}\label{prop:cdQtreedec}
   With the notation of Theorem~\ref{thm:abstDec}, 
   let 
   \begin{equation*}
       M:=\max\{\ccd_\Q(W_{V_t})\mid t\in V(\euT)\}\,\,\text{and}\,\, N:=\max\{\ccd_\Q(W_{V_t}\cap W_{V_{t'}})\mid \{t,t'\}\in E(\euT)\}.
   \end{equation*}
   Then
    \begin{equation*}
        M\leq \ccd_\Q(W)\leq \max\{M,N+1\}.
    \end{equation*}
\end{prop}
\begin{proof}
   It is a straightforward consequence of the Mayer--Vietoris sequence associated to the decomposition of~$W$ as the fundamental group of a tree of groups as in Theorem~\ref{thm:abstDec}.
\end{proof}

Since finite groups are the only groups that have vanishing rational cohomological dimension (cf.~Fact~\ref{fact:discoho}(3)), combining Example~\ref{ex:treedec1} and Proposition~\ref{prop:cdQtreedec} one deduces the following.
\begin{cor}\label{cor:cdQ2conn}
   Let $(W,S)$ be an infinite Coxeter group with connected presentation diagram~$\Gamma=\Gamma_\infty(W,S)$. Let $\{\Gamma_i=(V_i,E_i)\}_{i\in I}$ be the collection of all maximal $2$-connected components of $\Gamma$. Let $M=\max\{\ccd_\Q(W_{V_i})\mid i\in I\}$. Then $$\ccd_\Q(W)=\begin{cases}
       M,&\quad  \text{if $M\neq0$,}\hfil \\
       \hfil 1,&\quad \hfil  \text{if $M=0$.}\hfil
   \end{cases}$$
\end{cor}

\bibliographystyle{plain}
\bibliography{coxbuild}
\end{document}